\documentclass[reqno]{amsart}
\usepackage{graphicx}
\usepackage{amsmath}
\usepackage{amsthm}
\usepackage{mathtools}
\usepackage{amsfonts}
\usepackage{amssymb}
\usepackage{faktor}
\usepackage{wrapfig}
\usepackage{epstopdf}
\usepackage{caption}
\usepackage{url}
\usepackage{mathtools} 
\usepackage{float}
\usepackage{caption}
\usepackage{centernot}
\usepackage{dsfont}
\usepackage{mathtools}

\DeclarePairedDelimiter{\floor}{\lfloor}{\rfloor}
\usepackage{tikz}
\usetikzlibrary{matrix,arrows,decorations.pathmorphing}
\usepackage{pgf}
\usepackage{tikz-cd}

\title[Heegaard Floer homology and knots determined by complements]{
Heegaard Floer homology and knots determined by their complements}

\author{Fyodor Gainullin}
\address{}
\email{fyodor.gainullin@gmail.com}

\makeatletter
\DeclareRobustCommand\widecheck[1]{{\mathpalette\@widecheck{#1}}}
\def\@widecheck#1#2{%
    \setbox\z@\hbox{\m@th$#1#2$}%
    \setbox\tw@\hbox{\m@th$#1%
       \widehat{%
          \vrule\@width\z@\@height\ht\z@
          \vrule\@height\z@\@width\wd\z@}$}%
    \dp\tw@-\ht\z@
    \@tempdima\ht\z@ \advance\@tempdima2\ht\tw@ \divide\@tempdima\thr@@
    \setbox\tw@\hbox{%
       \raise\@tempdima\hbox{\scalebox{1}[-1]{\lower\@tempdima\box
\tw@}}}%
    {\ooalign{\box\tw@ \cr \box\z@}}}
\makeatother

\newcommand{\cA}{A^+_k(K)}
\newcommand{\cAA}{\mathcal{A}^+_{i, p/q}(K)}
\newcommand{\cB}{B^+}
\newcommand{\cBB}{\mathcal{B}^+}
\newcommand{\cv}{v_k}
\newcommand{\ch}{h_k}
\newcommand{\hv}{\boldsymbol{v}_k}
\newcommand{\hvti}{\boldsymbol{\widetilde{v}}_k}
\newcommand{\hvt}{\boldsymbol{v}_k^T}

\newcommand{\hh}{\boldsymbol{h}_k}
\newcommand{\hhti}{\boldsymbol{\widetilde{h}}_k}
\newcommand{\hht}{\boldsymbol{h}_k^T}
\newcommand{\hA}{\boldsymbol{A}^+_k(K)}
\newcommand{\hAo}{\boldsymbol{A}^+_0(K)}
\newcommand{\hB}{\boldsymbol{B}^+}
\newcommand{\T}{\mathcal{T}^+_d}

\newcommand{\Ta}{\mathcal{T}^+}
\newcommand{\hAA}{\mathbb{A}^+_{i,p/q}(K)}
\newcommand{\hBB}{\mathbb{B}^+}
\newcommand{\hAT}{\boldsymbol{A}^T_k(K)}
\newcommand{\hBT}{\boldsymbol{B}^T}
\newcommand{\hATo}{\boldsymbol{A}^T_0(K)}
\newcommand{\hAr}{\boldsymbol{A}^{red}_k(K)}
\newcommand{\hBr}{\boldsymbol{B}^{red}}

\newcommand{\hAAT}{\mathbb{A}^T_{i,p/q}(K)}
\newcommand{\hAAr}{\mathbb{A}^{red}_{i,p/q}(K)}
\newcommand{\hBBT}{\mathbb{B}^T}
\newcommand{\hBBr}{\mathbb{B}^{red}}
\newcommand{\cD}{D^+_{i,p/q}}
\newcommand{\hD}{\boldsymbol{D}^+_{i,p/q}}
\newcommand{\hDT}{\boldsymbol{D}^T_{i,p/q}}
\newcommand{\hDr}{\boldsymbol{D}^{red}_{i,p/q}}
\newcommand{\hDti}{\boldsymbol{\widetilde{D}}^+_{i,p/q}}
\newcommand{\cAn}{A^+_{\floor{\frac{i+pn}{q}}}(K)}

\newcommand{\ZZ}{\mathbb{Z}}
\newcommand{\QQ}{\mathbb{Q}}
\newcommand{\RR}{\mathbb{R}}
\newcommand{\MC}{\mathbb{X}^+_{i, p/q}}

\newcommand{\FF}{\mathbb{F}}
\newcommand{\FR}{\mathbb{F}[U]}
\newcommand{\KS}{Y_{p/q}(K)}
\newcommand{\dU}{d(L(p,q),i)}

\newcommand{\hATn}{\boldsymbol{A}^T_{\floor{\frac{i+pn}{q}}}(K)}
\newcommand{\hArn}{\boldsymbol{A}^{red}_{\floor{\frac{i+pn}{q}}}(K)}

\newcommand{\HF}{HF^+(\KS ,i)}

\newcommand{\Sp}{Spin$^c$ structure}
\newcommand{\Sps}{Spin$^c$ structures}

\newenvironment{claim}[1]{\par\noindent\underline{Claim:}\space#1}{}
\newenvironment{claimproof}[1]{\par\noindent\underline{Proof:}\space#1}{\hfill $\blacksquare$}

\def\co{\colon\thinspace}

\newtheorem{theorem}{Theorem}

\newtheorem{lemma}[theorem]{Lemma}

\newtheorem{prop}[theorem]{Proposition}

\newtheorem{cor}[theorem]{Corollary}

\newtheorem{conj}[theorem]{Conjecture}

\numberwithin{theorem}{section}

\begin{document}

\begin{abstract}
In this paper we investigate the question of when different surgeries on a knot can produce identical manifolds. We show that given a knot in a homology sphere, unless the knot is quite special, there is a bound on the number of slopes that can produce a fixed manifold that depends only on this fixed manifold and the homology sphere the knot is in. By finding a different bound on the number of slopes, we show that non-null-homologous knots in certain homology $\RR P^3$'s are determined by their complements. We also prove the surgery characterisation of the unknot for null-homologous knots in $L$-spaces. This leads to showing that all knots in some lens spaces are determined by their complements. Finally, we establish that knots of genus greater than $1$ in the Brieskorn sphere $\Sigma(2,3,7)$ are also determined by their complements.
\end{abstract}

\maketitle

\section{Introduction}
\label{sec:intro}

Dehn surgery is an important and widely used technique for constructing 3--manifolds, yet many natural questions about it are still unanswered. For example, given a knot $K$ in a manifold $Y$, how many different surgeries on $K$ can produce a fixed manifold $Z$? Conjecturally, in generic circumstances, the answer is $1$. More precisely, we have

\begin{conj}[Cosmetic surgery conjecture, see {\cite[Conjecture 6.1]{gordonICMproceedings}, \cite[Problem 1.81(A)]{kirbyProblems}} and {\cite[Conjecture 1.1]{NiWu}}]
Let $K$ be a knot in a closed connected orientable 3--manifold $Y$, such that the exterior of $K$ is irreducible and not homeomorphic to the solid torus. Suppose there are two different slopes $r_1$ and $r_2$, such that there is an orientation preserving homeomorphism between $Y_{r_1}(K)$ and $Y_{r_2}(K)$. Then the slopes $r_1$ and $r_2$ are equivalent.
\label{conj-cosmetic}
\end{conj}

We call two slopes equivalent if there is a homeomorphism of the knot exterior taking one to the other. If there are two distinct surgeries on $K$ (with inequivalent slopes) that produce the same oriented manifolds, then we call such surgeries purely cosmetic.

Another very natural question about knots in 3--manifolds is whether knots are determined by their complements. In other words, given two distinct knots $K_1, K_2 \subset Y$, can there exist an orientation-preserving homeomorphism between $Y \setminus K_1$ and $Y\setminus K_2$? We remark that by `distinct' here we mean that there is no orientation-preserving homeomorphism of $Y$ taking $K_1$ to $K_2$.

Given a knot $K_1 \subset Y$, we say that $K_1$ is \slshape determined by its complement \upshape if there is no other knot $K_2 \subset Y$ such that there is an orientation-preserving homeomorphism between $Y\setminus K_1$ and $Y\setminus K_2$. We say that $K_1$ is \slshape strongly \upshape determined by its complement if the condition of the previous sentence holds without the insistence on the homeomorphism to be orientation-preserving.

By \cite{edwardsConcentricity} the question for complements is equivalent to the analogous question for exteriors. It is not difficult to see then that the question of whether a knot is determined by its complement can be reformulated in terms of Dehn surgery as follows. A knot $K\subset Y$ is determined by its complement if and only if the following condition holds. If a surgery of some slope $r$ on $K$ gives $Y$, then $r$ is equivalent to the meridian of $K$.

Knots in $S^3$ are determined by their complements \cite{gordonLueckeKnotComplement} (but not strongly, as there exist chiral knots). Apart from some obvious ones, no examples of knots that are not determined by their complements have been exhibited. Thus the following conjecture seems natural

\begin{conj}[Knot complement conjecture, see {\cite[Conjecture 6.2]{gordonICMproceedings}, \cite[Problem 1.81(D)]{kirbyProblems} \cite[Conjecture 6.2]{boyerSurgerySurvey}}]
Let $K$ be a knot in a closed connected orientable 3--manifold $Y$, such that the exterior of $K$ is irreducible and not homeomorphic to the solid torus. Suppose there is a non-trivial slope $r$ such that there is an orientation preserving homeomorphism between $Y_{r}(K)$ and $Y$. Then $r$ is equivalent to the meridian of $K$.
\label{conj-complement}
\end{conj}

We remark that dealing with equivalent slopes may seem complicated, but this issue does not arise at all if we can show that for no slope $r$ (other than the meridian) we have $Y_{r}(K) \cong Y$. This not only shows that $K$ is determined by its complement, but also that there is a unique slope which is equivalent to the meridian (i.e. the meridian itself). In all statements for which we will be able to show that a knot is determined by its complement, this will be the case.

Now a knot $K \subset Y$ is strongly determined by its complement if and only if it is determined by its complement and the following condition holds. If $Y_r(K)$ is homeomorphic to $-Y$ by an orientation-preserving homeomorphism, then there is an orientation-reversing homeomorphism of the exterior of $K$ that takes the meridian to $r$. For example, achiral knots in $S^3$ are strongly determined by their complements.

It is not difficult to see that Conjectures \ref{conj-cosmetic} and \ref{conj-complement} are, in fact, equivalent. However, they are not equivalent if we concentrate on a given manifold. In other words, if $Y$ in the conjectures is fixed, then they are genuinely different (Conjecture \ref{conj-cosmetic} implies Conjecture \ref{conj-complement}).

\

Conjecture \ref{conj-cosmetic} is wide-open. In contrast to Conjecture \ref{conj-complement} it is not even proven for knots in $S^3$. However, in \cite{NiWu} Ni and Wu (generalising some results in \cite{OSzRatSurg}) used Heegaard Floer homology with great success to address the cosmetic surgery conjecture in $S^3$ (or other $L$-space homology spheres). They have been able to show that

\begin{itemize}
\item many manifolds (including all Seifert fibred spaces) cannot be results of purely cosmetic surgery;
\item there are at most $2$ slopes on a knot that can yield the same (oriented) manifold by surgery and they are negatives of each other;
\item if $p/q$ is a purely cosmetic surgery slope, then $q^2 \equiv -1\ (\mbox{mod } p)$;
\item knots that admit purely cosmetic surgery satisfy certain conditions on their knot Floer homology.
\end{itemize}

In fact, it can be shown that any knot that admits a Seifert fibred surgery satisfies the cosmetic surgery conjecture (see \cite{gainullin}).

Boyer and Lines ruled out cosmetic surgeries on many knots in homology spheres. They showed \cite[Proposition 5.1]{boyerLines} that knots with $\Delta_K''(1) \neq 0$ satisfy the cosmetic surgery conjecture. Here $\Delta_K$ is the Alexander polynomial of $K$ normalised so as to be symmetric and satisfy $\Delta_K(1) = 1$.

One might try to approach Conjectures \ref{conj-cosmetic} and \ref{conj-complement} by trying to at least find some bound on the number of slopes on a knot that can yield the same manifold. (The cosmetic surgery conjecture states that this bound is $1$.) In fact, it follows from the work of Cooper and Lackenby in \cite{cooperLackenby} that, given two manifolds $Y$ and $Z$, there are only finitely many slopes $\alpha$ such that there exists a hyperbolic knot $K\subset Y$ with $Y_{\alpha}(K) = Z$.

Our first result says that given a knot $K$ in any homology sphere $Y$, if the Heegaard Floer homology of a 3--manifold $Z$ satisfies a certain property, then there is a bound on the number of slopes that can produce $Z$ that only depends on the first homology of $Z$. More precisely, we have

{
\renewcommand{\thetheorem}{\ref{theorem-Z-special}}
\begin{theorem}
Let $K$ be a knot in a homology sphere $Y$. Let $Z$ be a rational homology sphere with $|H_1(Z)| = p$ such that $p$ does not divide $\chi(HF_{red}(Z))$. Suppose that there exist $q_1$, $q_2$ such that
$$
Z = Y_{p/q_1}(K) = Y_{p/q_2}(K).
$$
Then there is no multiple of $p$ between $q_1$ and $q_2$. In particular, there are at most $\phi(p)$ surgeries on $K$ that give $Z$.
\label{theorem-Z-special}
\end{theorem}
\addtocounter{theorem}{-1}
}

Here $\phi$ denotes the Euler's totient function.

In particular, if $Z$ is a homology $\RR P^3$ (i.e. $|H_1(Z)|=2$) whose order of reduced Floer homology is odd, then it can be obtained by at most one surgery on any fixed knot in any homology sphere. A standard homological argument then gives

{
\renewcommand{\thetheorem}{\ref{cor-z-special-complement}}
\begin{cor}
Let $Z$ be a closed connected oriented manifold with $|H_1(Z)| = 2$. Suppose that $\dim(HF_{red}(Z))$ is odd. Then non-null-homologous knots in $Z$ are determined by their complements.
\label{cor-z-special-complement}
\end{cor}
\addtocounter{theorem}{-1}
}

Spaces that satisfy the conditions of Theorem \ref{theorem-Z-special} exist, some of which we exhibit in

{
\renewcommand{\thetheorem}{\ref{cor-z-special}}
\begin{cor}
Let $Z^1_m = S^2((3,-1),(2,1),(6m-2,-m))$ for odd $m \geq 3$ and $Z^2_n$ be the result of $2/n$-surgery on the figure-eight knot for any odd $n$. If $K$ is a knot in a homology sphere that gives one of $Z^1_m$ or $Z^2_n$ by surgery of some slope, then such surgery slope is unique.
\label{cor-z-special}
\end{cor}
\addtocounter{theorem}{-1}
}

Results of Ni and Wu show that Heegaard Floer homology is a relatively good invariant when restricted to the set of manifolds obtained by surgery on a fixed knot in $S^3$ in the sense that at most finitely many (i.e. $2$) of them can have the same Heegaard Floer homology (examples, when $2$ different surgeries have the same Heegaard Floer homology, do occur -- see \cite[Section 9]{OSzRatSurg}).

Let $K$ be a knot in an arbitrary homology sphere. We show that, unless $K$ is very special, the set of Heegaard Floer homologies of spaces obtained by surgery on $K$ still contains at most finitely many repetitions. In the statement below, subscripts $e$ and $o$ stand for even and odd (respectively) parts of homology groups. The rest of the notation will be explained in the next section.

For a homology sphere $Y$ and a rational homology sphere $Z$, define
$$
N(Y,Z) = 2|H_1(Z)|\dim(HF_{red}(Y))+\dim(HF_{red}(Z)).
$$

{
\renewcommand{\thetheorem}{\ref{theorem-K-special}}
\begin{theorem}
Let $Y$ be a non-$L$-space homology sphere, $Z$ be a rational homology sphere and $K \subset Y$ be a knot and suppose there are coprime integers $p,q$ such that $Z = Y_{p/q}(K)$.

If $|q|>N(Y,Z)$, then

\begin{itemize}
\item $V_0(K) = 0$;
\item $\Delta_K \equiv 1$;
\item $\dim(\hAr_e) = \dim(HF_{red}(Y)_{e})$ for all $k$;
\item $\dim(\hAr_o) = \dim(HF_{red}(Y)_{o})$ for all $k$.
\end{itemize}
\label{theorem-K-special}
\end{theorem}
\addtocounter{theorem}{-1}
}

Note that if $Y$ is an $L$-space homology sphere, $K$ a knot in it and $Z = Y_{p/q}(K)$, then by \cite[Theorem 7]{gainullin} we have $|q| \leq p + \dim(HF_{red}(Z))$\footnote{In \cite{gainullin} the theorem is stated for $Y = S^3$, but the same proof is valid for arbitrary $L$-space homology spheres.}. Therefore we lose nothing by assuming that $Y$ is not an $L$-space.

Theorem \ref{theorem-K-special} shows that for a knot to have a surgery not satisfying the bound $|q| \leq N(Y,Z)$, the knot has to be quite `special', i.e. satisfy all the four of the bullet points listed in the statement of the theorem. In the next proposition, we show that even for such `special' knots that have genus $>1$ we can provide an alternative upper bound on $q$.

For a rational homology sphere $Z$, define 

$$
\widehat{D}(Z) = \max\{\mathrm{deg}(z)-d(Z,\mathfrak{t})\ |\ \text{homogeneous } z\in HF_{red}(Z,\mathfrak{t}),\mathfrak{t}\in \mathrm{Spin}^c(Z)\}
$$

and

$$
\widecheck{D}(Z) = \min\{\mathrm{deg}(z)-d(Z,\mathfrak{t})\ |\ \text{homogeneous } z\in HF_{red}(Z,\mathfrak{t}),\mathfrak{t}\in \mathrm{Spin}^c(Z)\},
$$

where by $\mathrm{deg}(z)$ we understand the absolute grading of $z$.

{
\renewcommand{\thetheorem}{\ref{prop-excep-K-genus}}
\begin{prop}
Let $K\subset Y$ be a knot such that $Z = Y_{p/q}(K)$ for $p,q>0$ and $q>N(Y,Z)$. Suppose the genus of $K$ is larger than one. Then
$$
\floor{q/p} \leq \frac{\widehat{D}(Z)-\widecheck{D}(Y)}{2}.
$$
\label{prop-excep-K-genus}
\end{prop}
\addtocounter{theorem}{-1}
}

Thus if infinitely many spaces obtained by surgery on a knot $K$ in a homology sphere have the same Heegaard Floer homology then $K$ has genus $1$ and trivial Alexander polynomial.

\

In \cite{KMOS} Kronheimer, Mrowka, Ozsv\'ath and Szab\'o prove the following `surgery characterisation of the unknot'.

\begin{theorem}[{\cite[Theorem 1.1]{KMOS}}]
Let $U$ denote the unknot in $S^3$, and let $K$ be any knot. If there is an orientation-preserving diffeomorphism $S^3_r(K) \cong S^3_r(U)$ for some rational number $r$, then $K = U$.
\label{theorem-KMOS}
\end{theorem}

Clearly, this Theorem provides an alternative proof of the fact that knots in $S^3$ are determined by their complements. The proof has been adapted to the setting of Heegaard Floer homology in \cite{OSzLecturesHF} (see also \cite{manion}). Crucial for the proof is the fact that $S^3$ is an $L$-space, i.e. it has `the smallest possible' Heegaard Floer homology.

The Brieskorn space $\Sigma(2,3,7)$ has perhaps the simplest possible Heegaard Floer homology a non-$L$-space can have. We use this to prove

{
\renewcommand{\thetheorem}{\ref{theorem-brieskorn}}
\begin{theorem}
Knots of genus larger than $1$ in the Brieskorn sphere $\Sigma(2,3,7)$ are determined by their complements. Moreover, if $K\subset \Sigma(2,3,7)$ is a counterexample to Conjecture \ref{conj-complement} then the surgery slope is integral, $\widehat{HFK}(\Sigma(2,3,7),K,1)$ has dimension $2$ and its generators lie in different $\ZZ_2$-gradings.

Non-fibred knots of genus larger than $1$ in $\Sigma(2,3,7)$ are strongly determined by their complements.
\label{theorem-brieskorn}
\end{theorem}
\addtocounter{theorem}{-1}
}

Returning to $L$-spaces, we show that Theorem \ref{theorem-KMOS} admits a generalisation as follows.

{
\renewcommand{\thetheorem}{\ref{theorem-unknot-characterisation}}
\begin{theorem}
Let $Y$ be an $L$-space and $K\subset Y$ a null-homologous knot. Suppose that
$$
HF^+(Y_{p/q}(K)) \cong HF^+(Y \# L(p,q)).
$$

Then $K$ is the unknot.

In particular, null-homologous knots in $L$-spaces are determined by their complements.
\label{theorem-unknot-characterisation}
\end{theorem}
\addtocounter{theorem}{-1}
}

In the statement above, by $L(p, q)$ we understand the result of $p/q$ surgery on the unknot in $S^3$.

Shortly after the first version of this preprint was published, Ravelomanana \cite{ravelomananaComplement} published his proof of the fact that knots in $L$-space integral homology spheres are determined by their complements, thus reproving a part of our Theorem \ref{theorem-unknot-characterisation}. In a later version of his paper, he was also able to show that knots in $\Sigma(2,3,5)$ are (in our terminology) strongly determined by their complements.

Lens spaces are $L$-spaces (indeed, they are the reason for the name), so it follows that null-homologous knots in lens spaces are determined by their complements. In fact, Mauricio has proven the above Theorem for integral slopes in \cite{mauricioThesis}, so, coupled with the Cyclic Surgery Theorem of \cite{CGLS}, Mauricio's result implies that null-homologous non-torus knots in lens spaces are determined by their complements (though Mauricio does not phrase it in this way).

Some homological arguments allow us to prove a stronger result for lens spaces. First, we need to fix some notation. Note that we will be making some arbitrary choices. Suppose $K$ is a knot in a lens space $L = L(p,q)$ and view $L$ as a union of two Heegaard solid tori $V$ and $W$. Isotope $K$ into $W$ and fix thus obtained isotopy class of $K$ in $W$. Now that $K$ is viewed as a fixed knot in $W$, it has a well-defined winding number $w$ in $W$ (i.e. the algebraic intersection number of $K$ with a meridional disc of $W$ -- it does not make sense if we allow $K$ to leave $W$). Embed $W$ into $S^3$ in a standard way. This endows both $K$ and $W$ with a preferred longitude (note, again, that this only makes sense after we fix the embedding, which is chosen arbitrarily). We use thus obtained longitude of $K$ to identify slopes with rational numbers.

Note that even though $w$ was fixed rather arbitrarily, it features in the Corollary below. This is because its remainder modulo $p$ is well-defined and the slope $n$ in the statement also depends upon the choices above.

Then we have the following result

{
\renewcommand{\thetheorem}{\ref{cor-lens-space-complements}}
\begin{cor}
If $p$ is square-free, then all knots in $L = L(p,q)$ are determined by their complements.

More precisely, let $K$ be a knot whose exterior is not a solid torus and such that a non-trivial surgery on it gives $L$. Then the exterior of $K$ is not Seifert fibred, $p|w^2$ and the surgery slope, $n$, is an integer that satisfies the following (with some choice of sign):
$$
n = -q\frac{w^2}{p}\pm 1.
$$

Moreover, there is at most one such slope (i.e. we can choose either $+$ or $-$ but not both in the equation above).
\label{cor-lens-space-complements}
\end{cor}
\addtocounter{theorem}{-1}
}

\section*{Acknowledgements}

I would like to thank Andras Juhasz and Tye Lidman for valuable comments about the earlier versions of this paper. I am very grateful to Tom Hockenhull and Marco Marengon for persistently suggesting I should write some of the results in this paper up. Special thanks go to Tye Lidman and Duncan McCoy who found a mistake in the previous version of this paper. I would also like to thank my supervisor Dorothy Buck for her continued support over the course of my Ph.D. studies.

Finally, I am very grateful to the anonymous referee. They have been very careful in reading the drafts of this paper and their detailed comments and suggestions improved the paper very substantially.

\section{Review of the mapping cone formula}
\label{sec:mapCone}

In this section, we review the mapping cone formula of \cite[Theorem 1.1]{OSzRatSurg}. We use notation largely similar to that of Ni and Wu in \cite{NiWu}.

Given a knot $K$ in a homology sphere $Y$ we can associate to it a doubly-pointed Heegaard diagram as in \cite{OSzKnotInv}. We define a complex $C = CFK^{\infty}(Y,K)$ generated (over an arbitrary field $\FF$) by elements of the form $[\boldsymbol{x}, i , j]$, where $\boldsymbol{x}$ is an `intersection point' of the Heegaard diagram (as defined in \cite{OSzKnotInv}) and $(i,j) \in \ZZ \times \ZZ$. Generators of $C$ are not all triples $[\boldsymbol{x},i,j]$, but only those that satisfy a certain condition. The differential on $C$ does not increase either $i$ or $j$, so $C$ is doubly-filtered by the pair $(i,j) \in \ZZ \times \ZZ$. The doubly-filtered chain homotopy type of this complex is a knot invariant \cite[Theorem 3.1]{OSzKnotInv}.

By \cite[Lemma 4.5]{rasmussenThesis} the complex $C$ is homotopy equivalent (as a filtered complex) to a complex for which all filtration-preserving differentials are trivial. In other words, at each filtration level, we replace the group, viewed as a chain complex with the filtration preserving differential, by its homology. From now on we work with this, \slshape reduced \upshape complex.

The complex $C$ is invariant under the shift by the vector $(-1,-1)$. Thus there is an action of a formal variable $U$ on $C$, which is simply the translation by the vector $(-1,-1)$. In other words, the group at the filtration level $(i,j)$ is the same as the one at the filtration level $(i-1, j-1)$ and $U$ is the identity map from the first one to the second. Of course, $U$ is a chain map. In $C$ the map $U$ is invertible (but note that it will not be in various subcomplexes and quotients), so $C$ is an $\mathbb{F}[U, U^{-1}]$-module.

This means that as an $\mathbb{F}[U, U^{-1}]$-module $C$ is generated by the elements with the first filtration level $i = 0$. In the reduced complex, the group at filtration level $(0, j)$ is denoted $\widehat{HFK}(Y,K,j)$ and is known as the knot Floer homology of $K$ at Alexander grading $j$.

\begin{figure}
\centering
\includegraphics[scale=0.6, clip = true, trim = 100 270 315 250]{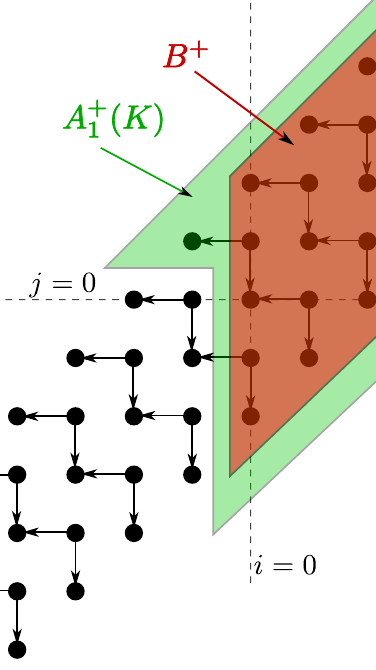}
\caption{Schematic representation of (a part of) the complex $C$ (for some genus-$2$ knot). Dots represent groups at various filtration levels and arrows stand for components of the differential. The part shaded in green (including the red part over it) is the complex $A^+_1(K)$. The part shaded red represents $B^+$.}
\label{figure-complex}
\end{figure}

The complex $C$ possesses an absolute $\QQ$-grading and a relative $\ZZ$-grading, i.e. differences of absolute $\QQ$-gradings of elements of $C$ are integers. In fact, the complex $C$ is the complex used to compute the ($\infty$-flavour of the) Heegaard Floer homology of $Y$, the knot provides an additional filtration for it. By grading the Heegaard Floer homology of $Y$ (as in \cite{OSzAbsGr}) we obtain the grading on $C$. The map $U$ decreases this grading by 2.

Using the filtration on $C$ we can define the following quotients (see Figure \ref{figure-complex}).
$$
\cA = C\{i\geq 0 \mbox{ or } j \geq k\}, \mbox{\ \ } k \in \mathbb{Z}
$$
and
$$
\cB = C\{i \geq 0\} \cong CF^+(Y).
$$

We also define two chain maps, $\cv, \ch\co \cA \to \cB$. The first one is just the projection (i.e. it sends to zero all generators with $i < 0$ and acts as the identity map for everything else). The second one is the composition of three maps: firstly we project to $C\{j \geq k\}$, then we multiply by $U^k$ (this shifts everything by the vector $(-k,-k)$) and finally we apply a chain homotopy equivalence that identifies $C\{j \geq 0\}$ with $C\{i \geq 0\}$. Such a chain homotopy equivalence exists because the two complexes both represent $CF^+(Y)$ and by general theory \cite{OSzKnotInv} there is a chain homotopy equivalence between them, induced by the moves between the Heegaard diagrams. We usually do not know the explicit form of this chain homotopy equivalence.

Knot Floer homology detects the knot genus. It does so in the following way \cite[Theorem 3.1]{niThurstonNorm}.

\begin{theorem}[Ni]
Let $Y$ be a homology sphere and $K\subset Y$ a knot. 

Then $g(K) = \max\{j \in \ZZ\ |\ \widehat{HFK}(Y,K,j) \neq 0\}$.
\label{theorem-genus_detect}
\end{theorem}

From this (together with symmetries of $C$) we can see that the maps $\cv$ (respectively $\ch$) are isomorphisms if $k \geq g$ (respectively $k \leq -g$). For example, Figure \ref{figure-complex} represents some knot of genus $2$.

We define chain complexes
$$
\cAA = \bigoplus_{n \in \mathbb{Z}}(n,\cAn), \ \cBB = \bigoplus_{n \in \mathbb{Z}}(n,\cB).
$$

The first entry in the brackets here is simply a label used to distinguish different copies of the same group. 
There is a chain map $\cD$ from $\cAA$ to $\cBB$ defined by taking sums of all maps $\cv, \ch$ with appropriate domains and requiring that the map $\cv$ goes to the group with the same label $n$ and $\ch$ increases the label by 1. Explicitly $\cD(\{(k,a_k)\}_{k \in \ZZ}) = \{(k,b_k)\}_{k \in \ZZ}$, where $b_k = v^+_{\floor{\frac{i+pk}{q}}}(a_k)+h^+_{\floor{\frac{i+p(k-1)}{q}}}(a_{k-1})$ -- see Figure \ref{cone_complex}.

\begin{figure}
\includegraphics[scale=0.8, clip = true, trim = 65 725 60 0]{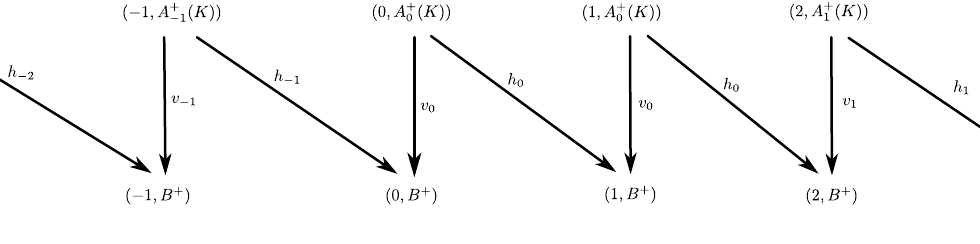}
\caption{Schematic representation of the portion of the map $\cD$ for $i = 0$ and $p/q = 2/3$.}
\label{cone_complex}
\end{figure}

Each of $\cA$ and $\cB$ inherits a relative $\ZZ$-grading from the one on $C$. Let $\MC$ denote the mapping cone of $\cD$. We fix a relative $\ZZ$-grading on the whole of it by requiring that the maps $\cv, \ch$ (and so $\cD$) decrease it by 1. The following is proven in \cite{OSzRatSurg}.

\begin{theorem}[Ozsv\'ath-Szab\'o]
\label{mapping cone}
There is a relatively graded isomorphism of $\FR$-modules
$$
H_*(\MC ) \cong HF^+( \KS , i ).
$$
\end{theorem}

The index $i$ in $HF^+( \KS , i )$ stands for a \Sp. The numbering of \Sps\ we refer to is defined in \cite{OSzRatSurg}, but we do not need precise details of how to obtain this numbering for our purposes.

We can also determine the absolute grading on the mapping cone. The group $\cBB$ is independent of the knot. Now if we insist that the absolute grading on the mapping cone for the unknot should coincide with the grading of $HF^+$ of the surgery on it (i.e. $\dU+d(Y)$), this fixes the grading on $\cBB$. We then use this grading to fix the grading on $\MC$ for arbitrary knots -- this grading then is the correct grading, i.e. it coincides with the one $HF^+$ should have.

It seems quite complicated to understand the homology of the mapping cone of $\cD$ by direct inspection. Thus we pass to homology of the objects we introduced above. Specifically, let $\hA = H_*(\cA ), \ \hB = H_*(\cB ), \ \hAA = H_*(\cAA), \hBB = H_*(\cBB)$ and let $\hv, \hh, \hD$ denote the maps induced by $\cv, \ch, \cD$ (respectively) in homology.

When we talk about $\hAA$ as an absolutely graded group, we mean the grading that it inherits from the absolute grading of the mapping cone that we described above.

Recall that the short exact sequence

{\setlength\mathsurround{0pt}
\begin{center}
\begin{tikzcd}
0 \arrow{r}
&\cBB \arrow{r}{i}
&\MC \arrow{r}{j}
&\cAA \arrow{r}
&0
\end{tikzcd}
\end{center}
}

induces the exact triangle

{\setlength\mathsurround{0pt}
\begin{center}
\begin{equation}
\begin{tikzcd}
\hAA \arrow{r}{\hD}\arrow[leftarrow]{rd}{j_*}
&\hBB \arrow{d}{i_*}\\
&H_*(\MC ) \cong \HF.
\label{extriang}
\end{tikzcd}
\end{equation}
\end{center}
}

All maps in these sequences are $U$-equivariant.

Define $\T$ to be the graded $\FR$-module $\FF[U,U^{-1}]/U\cdot \FF[U]$ with (the equivalence class of) $1$ having grading $d$ and multiplication by $U$ decreasing the grading by $2$. Similarly, let $\mathcal{T}(N)$ be the submodule of $\T$ generated by $\{1, U^{-1}, \ldots , U^{-(N-1)}\}$. We omit the subscript $d$ if the absolute grading does not exist or is not relevant. However, even without the absolute grading, these groups are still relatively $\ZZ$-graded (by requiring that $U$ decreases the grading by $2$).

If $Z$ is a rational homology sphere, then $HF^+(Z, \mathfrak{s}) = \T \oplus HF_{red}(Z, \mathfrak{s})$, where $d = d(Z, \mathfrak{s})$ is the $d$-invariant (or the correction term) of $Z$ in the \Sp\ $\mathfrak{s}$ and $HF_{red}(Z, \mathfrak{s})$ is the reduced Floer homology of $Z$ in the same \Sp. The reduced Floer homology of $Z$ is the sum of reduced Floer homologies in all \Sps, which we denote
$$
HF_{red}(Z) = \bigoplus_{\mathfrak{s}\in\mathrm{Spin}^c(Z)}HF_{red}(Z, \mathfrak{s}).
$$

For each $\mathfrak{s}\in \mathrm{Spin}^c(Z)$, $HF_{red}(Z, \mathfrak{s})$ is a finitely generated $\FR$-module in the kernel of a large enough power of $U$, thus it has the form $\bigoplus_{i=1}^m \mathcal{T}(n_i) $ for some $n_i$.

We have that $\hA \cong \hAT \oplus \hAr$ and $\hB = \hBT \oplus \hBr$, where $\hAT \cong \Ta \cong \hBT$ and $\hAr$ and $\hBr$ are finitely generated-$\FR$ modules in the kernel of a large enough power of $U$. Define
$$
\hAAT = \bigoplus_{n \in \ZZ}(n,\hATn), \mbox{ } \hAAr = \bigoplus_{n \in \ZZ}(n,\hArn),
$$

$$
\hBBT = \bigoplus_{n \in \ZZ}(n, \hBT), \mbox{ } \hBBr = \bigoplus_{n \in \ZZ}(n,\hBr).
$$

We decompose the maps in a similar manner. Let $\hD = \hDT \oplus \hDr$, where the first map is the restriction of $\hD$ to $\hAAT$ and the second one is the restriction to $\hAAr$. Let $\hvt$ and $\hht$ be the restrictions of $\hv$ and $\hh$ respectively to $\hAT$. Then $\hDT$ is defined using $\hvt$ and $\hht$ in the same way as $\hD$ is defined using $\hv$ and $\hh$.

Notice that the images of $\hvt$ and $\hht$ are contained in $\hBT$ -- this is because they are $\FR$-module maps. In fact, since these maps are homogeneous and are isomorphisms for large enough gradings, they are multiplications by some powers $U^{V_k}$ and $U^{H_k}$ for $\hvt$ and $\hht$ respectively.

The following are some useful properties of $V_k$ and $H_k$; the proofs are completely analogous to the case of knots in $S^3$, see \cite{NiWu}:

\begin{itemize}
\item $V_k = H_{-k}$ for any $k \in \ZZ$;
\item $V_k \geq V_{k+1}$ and $H_k \leq H_{k+1}$ for any $k \in \ZZ$;
\item $V_k \to + \infty$ as $k \to - \infty$ and $H_k \to + \infty$ as $k \to + \infty$;
\item $V_k = 0$ for $k\geq g(K)$ and $H_k = 0$ for $k \leq -g(K)$.
\label{vhprop}
\end{itemize}

In other words, $V_k$ form a non-increasing unbounded sequence of non-negative numbers, which become zero at $g(K)$ or earlier and $H_k = V_{-k}$.

\section{Correction terms}
\label{sec:corTerms}

The next lemma essentially shows that when $p, q > 0$, the map $\cD$ becomes an isomorphism `at the ends', so in the mapping cone formula we only need to consider a finite central part.

\begin{lemma}
Fix a number $G \geq g(K)$. Let $p, q > 0$. Let $\hBB_{G^+}$ and $\hBB_{G^-}$ be the subgroups of $\hBB$ consisting of all $(n, \hB)$ with $n$ satisfying
$$
\left\lfloor \frac{i+pn}{q} \right\rfloor\geq G
$$

and

$$
\left\lfloor \frac{i+p(n-1)}{q} \right\rfloor\leq -G.
$$

respectively.

\



Let $\hAA_{G^+}$ and $\hAA_{G^-}$ be the subgroups of $\hAA$ consisting of all $(n, \hA)$ with $n$ satisfying
$$
\left\lfloor \frac{i+pn}{q} \right\rfloor \geq G
$$

and

$$
\left\lfloor \frac{i+pn}{q} \right\rfloor\leq -G
$$

respectively.


Then $\hD$ maps $\hAA_{G^{\pm}}$ isomorphically onto $\hBB_{G^{\pm}}$.
\label{lemma-iso-at-ends}
\end{lemma}
\begin{proof}
The cases $n\geq 0$ and $n\leq0$ are similar, so we will only consider $n\geq 0$.

First we want to show that the image of $\hD$ is all of $\hBB_{G^+}$. Suppose $\xi \in (n, \hB)$, $n\geq 0 \mbox{ and } \floor{\frac{i+pn}{q}}\geq G \geq g(K)$.

Note that for $k \geq g(K)$, $\hv$ is an isomorphism. Moreover, if we identify each $\hA$ with $\hB$ via $\hv$ for $k \geq g(K)$, then any fixed element of $\hA$ is in the kernel of $\hh$ for big enough $k$ and $\hh$ decreases the grading by any amount we want if $k$ is big enough.

Let $\eta_0 = \boldsymbol{v}^{-1}_{\floor{\frac{i+pn}{q}}}(\xi)\in (n, \boldsymbol{A}^+_{\floor{\frac{i+pn}{q}}})$. Define $\eta_m$ inductively by
$$
\eta_m = \boldsymbol{v}^{-1}_{\floor{\frac{i+(n+m)p}{q}}}(\boldsymbol{h}_{\floor{\frac{i+(n+m-1)p}{q}}}(\eta_{m-1})).
$$

By properties of $\hv$ and $\hh$ described above, for big enough $m$ we have $\eta_m = 0$. This shows that $\xi = \hD (\sum_k \eta_k)$ is in the image of $\hD$.

Now suppose $\eta \in \hAA_{G^+}$ is in the kernel of the restriction of $\hD$ to $\hAA_{G^+}$. Then the leftmost component must be in the kernel of the corresponding map $\hv$ -- a contradiction.
\end{proof}

In the previous section, we mentioned that it is possible to put an absolute grading on $\hBB$ that does not depend on the knot we consider. In the next lemma, we explicitly describe this grading. Note that since the relative grading is already fixed it is enough to put an absolute grading on any homogeneous element of $\hBB$.

\begin{lemma}
Let $Y$ be a homology sphere. Consider the mapping cone for the \Sp\ $i$. The grading of $1$ in $(0, \hBT)$ is $d(Y)+d(L(p,q),i)-1$.
\label{lemma-unknot-d}
\end{lemma}
\begin{proof}
As a result of $p/q$-surgery on the unknot we get $Y\#L(p,q)$, whose correction terms are $d(Y\#L(p,q),i) = d(Y)+d(L(p,q),i)$.

By Lemma \ref{lemma-iso-at-ends} applied to the unknot (which has genus $0$), the map $\hD$ is surjective for the unknot. Thus $HF^+(Y\#L(p,q)) \cong \ker(\hD)$.

Just as in \cite[proof of Proposition 1.6]{NiWu}, there is a tower in the kernel of $\hDT$ and the element $U^{-n}$ in this tower has $U^{-n}$ as a component in $(0, \hATo)$. This shows that $1$ in $(0, \hATo)$ has grading $d(Y\#L(p,q),i) = d(Y)+d(L(p,q),i)$, so (since $V_0 = 0$) $1$ in $(0, \hBT)$ has grading $d(Y)+d(L(p,q),i)-1$.
\end{proof}

Let $Y$ be a homology sphere. Recall that its Heegaard Floer homology possesses an absolute $\ZZ_2$-grading, defined to be $0$ on the tower part and be the reduction $\mathrm{mod}\ 2$ of the relative $\ZZ$-grading\footnote{In other words, every element of the tower has grading $0$ and the grading of an element is $1$ if and only if it has odd relative $\ZZ$-grading with some element of the tower.}. Decompose the Heegaard Floer homology of $Y$ in the following way: $HF^+(Y) \cong \Ta \bigoplus_{i=1}^l \mathcal{T}(n^+_i) \bigoplus_{i=1}^m \mathcal{T}(n^-_i)$, where $\mathcal{T}(n^+_i)$ (respectively $\mathcal{T}(n^-_i)$) lie in even (respectively odd) $\ZZ_2$-grading. The following Proposition may be seen as a generalisation of \cite[Proposition 1.6]{NiWu}.

\begin{prop}
With notation as above, suppose $Z = Y_{p/q}(K)$, for $p>0$, $q>0$. Then
\begin{equation}
d(Y) + d(L(p,q),i) - 2\max\{V_{\floor{\frac{i}{q}}}, H_{\floor{\frac{i-p}{q}}}\} - 2\max_j\{n^-_j\}\leq d(Z,i)
\label{d-ineq-lower}
\end{equation}

and

\begin{equation}
d(Z,i) \leq d(Y) + d(L(p,q),i) - 2\max\{V_{\floor{\frac{i}{q}}}, H_{\floor{\frac{i-p}{q}}}\}.
\label{d-ineq-upper}
\end{equation}
\label{prop-d-ineq}
\end{prop}
\begin{proof}
Since the grading of $\hBB$ is independent of the knot, by Lemma \ref{lemma-unknot-d} we have that $1$ in $(0, \hBT)$ has grading $d(Y)+d(L(p,q),i)-1$. As usual, the proof subdivides into two cases, depending on whether $V_{\floor{\frac{i}{q}}} = \max\{V_{\floor{\frac{i}{q}}}, H_{\floor{\frac{i-p}{q}}}\}$ or otherwise. The two cases are analogous, so we only consider the case $V_{\floor{\frac{i}{q}}} = \max\{V_{\floor{\frac{i}{q}}}, H_{\floor{\frac{i-p}{q}}}\}$.

Then, as in \cite[proof of Proposition 1.6]{NiWu} (or see \cite[Lemmas 12-13]{gainullin}) we can show that there is a tower in the kernel of $\hDT$, such that $U^{-n}$ in this tower has $U^{-n}$ as the component in $(0, \hATo)$. Suppose that $1$ in this tower has grading $d$. Then also $1$ in $(0, \hATo)$ has grading $d$ and thus $U^{-V_{\floor{\frac{i}{q}}}} \in (0, \hATo)$ has grading $d+2V_{\floor{\frac{i}{q}}}$.

On the other hand, $U^{-V_{\floor{\frac{i}{q}}}} \in (0, \hATo)$ is mapped to $1 \in (0, \hBT)$ by $\boldsymbol{v}^+_{\floor{\frac{i}{q}}}$, which has grading $-1$. So
$$
d+2V_{\floor{\frac{i}{q}}}-1 = d(Y)+d(L(p,q),i)-1,
$$
from which it follows that $d = d(Y)+d(L(p,q),i) - 2V_{\floor{\frac{i}{q}}}$.

By the exact triangle \eqref{extriang}, everything in the kernel of $\hD$ must be in the image of $j_*$. So in particular, the tower we identified in the kernel of $\hDT$ (and thus $\hD$) must be in the image of $j_*$. At high enough gradings only the elements of the tower in $\HF$ may hit the elements of the tower in the kernel of $\hD$. Since the maps in the triangle are $U$-equivariant, the tower in $\HF$ must be mapped onto the tower in the kernel of $\hD$. It follows that
$$
d(Y)+d(L(p,q),i) - 2V_{\floor{\frac{i}{q}}} = d \geq d(Z,i).
$$

This argument also shows that the map $j_*$ has submodule $\mathcal{T}(\frac{1}{2}(d-d(Z,i)))$ in its kernel and moreover this submodule lies in $\ZZ_2$-grading $0$.

Thus there has to be a submodule $\mathcal{T}(N)$ in $\hBB$ with $N\geq \frac{1}{2}(d-d(Z,i))$ such that it is not in the image of $\hD$. Moreover, it must have odd $\ZZ_2$-grading in $\hBB$.

However, the odd part of $\hBB$ is in the kernel of $U^{\max_j\{n^-_j\}}$, so $\max_j\{n^-_j\} \geq \frac{1}{2}(d-d(Z,i))$ and therefore
$$
d(Z,i) \geq d - 2\max_j\{n^-_j\} = d(Y)+d(L(p,q),i) - 2V_{\floor{\frac{i}{q}}} -  2\max_j\{n^-_j\}.
$$

This completes the proof in the case $V_{\floor{\frac{i}{q}}} \geq H_{\floor{\frac{i-p}{q}}}$. The other case is completely analogous.
\end{proof}

The following straightforward corollary may be of interest.

\begin{cor}
Let $Y$ be a positively oriented Seifert fibred homology sphere and $K\subset Y$ a knot. Suppose $Z = Y_{p/q}(K)$, where $p/q>0$. Then
$$
d(Z,i) = d(Y) + d(L(p,q),i) - 2\max\{V_{\floor{\frac{i}{q}}}, H_{\floor{\frac{i-p}{q}}}\}.
$$
\label{cor-d-inv-seif-pos}
\end{cor}
\begin{proof}
Positively oriented Seifert fibred homology spheres have $n^-_i = 0$ for all $i$ by \cite[Corollary 1.4]{OSzPlumbed}.
\end{proof}

\section{Surgery producing spaces with $p \not |\ \chi (HF_{red})$}
\label{sec:Z-special}

In this section, we want to prove Theorem \ref{theorem-Z-special}. We use the Casson-Walker invariant, normalised as in \cite{NiWu}. Our normalisation for the Alexander polynomial of a null-homologous knot in a rational homology also differs from that used in some other sources (in particular, \cite{walkerExtensionCasson}). Specifically, we require that the Alexander polynomial $\Delta_K$ of a null-homologous knot $K$ in a rational homology sphere $Y$ satisfies $\Delta_K(t) = \Delta_K(t^{-1})$ and $\Delta_K''(1) = |H_1(Y)|$.

To a rational homology sphere $Y$, the Casson-Walker invariant assigns a rational number $\lambda(Y)$. Two key properties we will need are as follows.

For a null-homologous knot $K$ in a rational homology sphere $Y$ we have (see \cite[Proposition 6.2]{walkerExtensionCasson} and note we are using slightly different normalisations)
\begin{equation}
\lambda(Y_{p/q}(K)) = \lambda(Y) + \lambda(L(p,q)) + \frac{q}{2p|H_1(Y)|}\Delta''_K(1).
\label{lambda-surg}
\end{equation}

The following formula appears in \cite[Theorem 3.3]{rustamov}:

\begin{equation}
|H_1(Y)|\lambda(Y) = \chi(HF_{red}(Y))-\frac{1}{2}\sum_{\mathfrak{s}\in \mathrm{Spin}^c(Y)}d(Y,\mathfrak{s}).
\label{lambda-hf}
\end{equation}

If we apply this formula to a lens space $L(p, q)$ we get

\begin{equation}
\sum_{\mathfrak{s}\in \mathrm{Spin}^c(L(p, q))}d(L(p,q),\mathfrak{s}) = -2p\lambda(L(p,q)).
\label{lens-sum}
\end{equation}



Another invariant we will briefly use is the Casson-Gordon invariant, $\tau$, which satisfies the following surgery formula. Suppose $W$ is an integral homology sphere and $K$ a knot in it. Then

\begin{equation}
\tau(Y_{p/q}(K)) = \tau(L(p,q)) - \sigma(K,p),
\label{tau-surg}
\end{equation}

where $\sigma(K,p)$ is a number depending only on $K$ and $p$.

Finally, both Casson-Walker and Casson-Gordon invariants of lens spaces can be expressed in terms of Dedekind sums. For our purposes, it is enough to know that a Dedekind sum assigns to a pair of coprime numbers $(p,q)$ a number $s(q,p)$. We have

\begin{equation}
\lambda(L(p,q)) = -\frac{1}{2}s(q,p),
\label{lambda-lens}
\end{equation}

and

\begin{equation}
\tau(L(p,q)) = -4ps(q,p).
\label{tau-lens}
\end{equation}

\begin{prop}
Let $Y$ be a homology sphere, $K \subset Y$ a knot and suppose there is a rational homology sphere $Z$ with
$$
Z = Y_{p/q_1}(K) = Y_{-p/q_2}(K),
$$

where $p,q_1,q_2>0$.

Then 
$$
\chi(HF_{red}(Z)) = p\chi(HF_{red}(Y)).
$$
\label{prop-chi-eq}
\end{prop}

\begin{proof}
By combining equations \eqref{tau-surg} and \eqref{tau-lens} we get $s(q_1,p) = s(-q_2,p)$. From this and equation \eqref{lambda-lens} we get $\lambda(L(p,q_1)) = \lambda(L(p,-q_2))$.

Now equation \eqref{lambda-surg} implies that 
$$
\frac{q_1}{2p|H_1(Z)|}\Delta''_K(1) = \frac{-q_2}{2p|H_1(Z)|}\Delta''_K(1),
$$

from which it follows that $\Delta''_K(1) = 0$.

Formula \eqref{d-ineq-upper} gives $d(Z,i) \leq d(Y) + d(L(p,q_1),i)$. If we can get $Z$ from $Y$ by $-\frac{p}{q_2}$-surgery, then by reversing orientations we see, that we can get $-Z$ from $-Y$ by $\frac{p}{q_2}$-surgery. Using formula \eqref{d-ineq-upper} then gives $d(-Z,i) \leq d(-Y) + d(L(p,q_2),i)$, which yields $-d(Z,i) \leq -d(Y) - d(L(p,-q_2),i) \Rightarrow d(Z,i) \geq d(Y) + d(L(p,-q_2),i)$.

Summing over all \Sps\ yields
\begin{equation}
\sum_{\mathfrak{s}\in \mathrm{Spin}^c(Z)}d(Z,\mathfrak{s}) \leq pd(Y) + \sum_{\mathfrak{s}\in \mathrm{Spin}^c(L(p, q_1))}d(L(p, q_1),\mathfrak{s}),
\label{d-sum-lower}
\end{equation}

and

\begin{equation}
\sum_{\mathfrak{s}\in \mathrm{Spin}^c(Z)}d(Z,\mathfrak{s}) \geq pd(Y) + \sum_{\mathfrak{s}\in \mathrm{Spin}^c(L(p, -q_2))}d(L(p,-q_2),\mathfrak{s}).
\label{d-sum-upper}
\end{equation}

By equation \eqref{lens-sum} the two last terms (sums of $d$-invariants for lens spaces) in the two equations above are equal.

The inequalities \eqref{d-sum-lower} and \eqref{d-sum-upper} now imply
\begin{multline}
\sum_{\mathfrak{s}\in \mathrm{Spin}^c(Z)}d(Z,\mathfrak{s}) = pd(Y) -2p\lambda(L(p,q_1)) = pd(Y)-2p\lambda(L(p,-q_2)).
\label{d-eq}
\end{multline}

Now using \eqref{lambda-hf} we get
$$
\chi(HF_{red}(Z)) - \frac{1}{2}\sum_{\mathfrak{s}\in \mathrm{Spin}^c(Z)}d(Z,\mathfrak{s}) = p\lambda(Z).
$$

Using \eqref{lambda-surg} and the fact that $\Delta''_K(1) = 0$ we obtain
$$
p\lambda(Z) = p\lambda(Y) + p\lambda(L(p,q_1)).
$$

Applying \eqref{lambda-hf} to $Y$ gives us
\begin{multline*}
p\lambda(Z) = p\chi(HF_{red}(Y)) - \frac{p}{2}d(Y) +p\lambda(L(p,q_1)) = \\
= p\chi(HF_{red}(Y)) -\frac{1}{2}(pd(Y) - 2p\lambda(L(p,q_1))).
\end{multline*}

Substituting from \eqref{d-eq} and combining the equalities we arrive at the following
$$
\chi(HF_{red}(Z)) - \frac{1}{2}\sum_{\mathfrak{s}\in \mathrm{Spin}^c(Z)}d(Z,\mathfrak{s}) = p\chi(HF_{red}(Y))- \frac{1}{2}\sum_{\mathfrak{s}\in \mathrm{Spin}^c(Z)}d(Z,\mathfrak{s}),
$$
from which the conclusion of the proposition follows.
\end{proof}

Since this proposition holds for arbitrary homology spheres, we have relative freedom to `change coordinates', i.e. to see a surgery on a knot in some homology sphere as a surgery on its dual in another homology sphere. This is the essence of what is going on in the next lemma.

\begin{lemma}
Let $K$ be a knot in a homology sphere $Y$ and suppose for some rational homology sphere $Z$
$$
Z = Y_{p/q_1}(K) = Y_{p/q_2}(K).
$$

Suppose further that there exists $k\in \ZZ$, such that $q_1<pk<q_2$. Then
$$
p|\chi(HF_{red}(Z)).
$$
\label{lemma-chi-divis}
\end{lemma}
\begin{proof}
Consider a homology sphere $Y_1$ given by
$$
Y_1 = Y_{1/k}(K).
$$

Let $K'$ be the surgery dual of $K$ in $Y_1$. Denote by $\mu$ the meridian of $K'$ and by $m$ and $l$ the meridian and the (preferred) longitude respectively of $K$. Longitudes of $K$ and $K'$ coincide. We view the curves $\mu$, $m$ and $l$ as slopes on the boundary of $Y\setminus \mathrm{nb}(K) = Y_1 \setminus \mathrm{nb}(K')$.

We have $\mu = m + kl$. So $pm+q_1l = p\mu + (q_1-pk)l$ and $pm+q_2l = p\mu + (q_2-pk)l$. Since $q_1-pk<0<q_2-pk$, this shows that $Z$ can be obtained by both positive and negative surgery on $K'$ in $Y_1$. Then by Proposition \ref{prop-chi-eq}
$$
\chi(HF_{red}(Z)) = p\chi(HF_{red}(Y_1))\Rightarrow p|\chi(HF_{red}(Z)).
$$
\end{proof}

We are now in position to prove

\begin{theorem}
Let $K$ be a knot in a homology sphere $Y$. Let $Z$ be a rational homology sphere with $|H_1(Z)| = p$ such that $p$ does not divide $\chi(HF_{red}(Z))$. Suppose that there exist $q_1$, $q_2$ such that
$$
Z = Y_{p/q_1}(K) = Y_{p/q_2}(K).
$$
Then there is no multiple of $p$ between $q_1$ and $q_2$. In particular, there are at most $\phi(p)$ surgeries on $K$ that give $Z$.
\label{theorem-Z-special}
\end{theorem}

\begin{proof}
If $\frac{p}{q_1}, \ldots, \frac{p}{q_N}$ are distinct slopes that give $Z$ by surgery on $K$ then by Lemma \ref{lemma-chi-divis} there is $k \in \ZZ$ such that $pk < q_i < p(k+1)$ for all $i$ (clearly the case $p = 1$ is vacuous). Since $q_i$ are coprime to $p$, the conclusion follows.
\end{proof}

If there are spaces that satisfy the condition of Theorem \ref{theorem-Z-special} and have order of homology $2$ then for any knot in any homology sphere there can be at most one slope that gives such a space by surgery (since $\phi(2) = 1$). Note also that $\dim(HF_{red}(Z))\equiv \chi(HF_{red}(Z)) \ (\mbox{mod }2)$, so the condition is then equivalent to $\dim(HF_{red}(Z))$ being odd. Such spaces do exist and the next corollary demonstrates some.

\begin{cor}
Let $Z^1_m = S^2((3,-1),(2,1),(6m-2,-m))$ for odd $m \geq 3$ and $Z^2_n$ be the result of $2/n$-surgery on the figure-eight knot for any odd $n$. If $K$ is a knot in a homology sphere that gives one of $Z^1_m$ or $Z^2_n$ by surgery of some slope, then such surgery slope is unique.
\label{cor-z-special}
\end{cor}

\begin{proof}
Note that $Z^1_m$ is the result of $2/m$-surgery on the right-handed trefoil. It is enough to show that the order of $HF_{red}(Z^1_m)$ or $HF_{red}(Z^2_n)$ is odd.

The trefoil has $V_0 = 1$ and $V_1 = 0$ and all $\hAr$ trivial (since it is a genus $1$ $L$-space knot). The dimension of its reduced Floer homology can be found using the formula of \cite[Corollary 3.6]{niZhangChar} (or see \cite[Proposition 16]{gainullin} for the same formula with notation as in this paper). In this case, the dimension is $m-2$ for $m\geq 3$, which is clearly odd.

Since the figure-eight knot is alternating, we can calculate its knot Floer homology from the Alexander polynomial and the signature (see \cite[Theorem 6.1]{OSzKnotGenusMutation}). This (after some calculations) shows that for the figure-eight $\hAo \cong \T \oplus \mathcal{T}_{d - 1}(1)$ for some $d$. We also have $V_0 = 0$ and $\hAr = 0$ for $k\neq 0$. Thus the dimension of the reduced Floer homology of surgery is equal to $n$ by \cite[Proposition 5.3]{NiWu}.
\end{proof}

\section{Bound on $q$ for knots that are not too exceptional}
\label{sec:K-special}

In this section, we prove Theorem \ref{theorem-K-special}. First, we deal with knots for which the sequence $\{V_k\}_{k\geq 0}$ is not identically zero.

%


\begin{lemma}
Let $K$ be a knot in a homology sphere $Y$ with $V_0>0$. Suppose $Z = Y_{p/q}(K)$ for $p \neq 0$. Then
$$
|q|\leq |p|+\frac{\dim(HF_{red}(Z))}{V_0}.
$$
\label{lemma-v0-not-0}
\end{lemma}
\begin{proof}
Let $|q| \geq |p|$. It follows from \cite[Lemma 13]{gainullin} for positive surgeries and \cite[Lemma 17]{gainullin} for negative surgeries that $HF_{red}(Z,i)$ contains a submodule isomorphic to $\mathcal{T}(V_0)^{\bigoplus n_i}$, where $n_i = \#\{j\in\ZZ\ |\ 0\leq j < |q|, j \equiv i \ (\mbox{mod } |p|)\}-1$.

So $\dim(HF_{red}(Z,i)) \geq n_iV_0$. Therefore 
$$
\dim(HF_{red}(Z))\geq V_0\sum_{i=0}^{|p|-1}n_i=V_0(|q|-|p|).
$$

The desired inequality now follows upon rearranging the terms.
\end{proof}

The mapping cone complex is sometimes unnecessarily large for our purposes. By using some elementary linear algebra contained in the next lemma, we want to be able to pass to a smaller complex when necessary.

\begin{lemma}
Let $T_1$, $T_2$, $R_1$ and $R_2$ be graded vector spaces and let $f\co T_1 \to T_2$, $g\co R_1 \to R_2$ and $h\co R_1 \to T_2$ be graded linear maps. Suppose that $f$ is surjective. Then the homology of the complex
\begin{center}
{\setlength\mathsurround{0pt}
\begin{equation*}
\begin{tikzcd}
0 \arrow{r}{} &T_1 \oplus R_1 \arrow{r}{D} &T_2\oplus R_2 \arrow{r}{} &0,
\end{tikzcd}
\end{equation*}
}
\end{center}
where $D$ is given by $D = \begin{pmatrix} f & h \\ 0 & g \\ \end{pmatrix}$, is isomorphic (as a graded vector space) to the direct sum of the kernel of the map $f$ and the homology of the complex

\begin{center}
{\setlength\mathsurround{0pt}
\begin{equation*}
\begin{tikzcd}
0 \arrow{r}{} &R_1 \arrow{r}{g} &R_2 \arrow{r}{} &0.
\end{tikzcd}
\end{equation*}
}
\end{center}
\label{lemma-vec-sp}
\end{lemma}

\begin{proof}
It is clear that the cokernels of the maps $D$ and $g$ are isomorphic as graded vector spaces -- indeed they are generated by the same elements not in the image of $g$.

We need to show that the kernels agree. Since $f$ is surjective, there is a graded map $f^*\co T_2 \to T_1$ such that $f\circ f^* = \mathrm{id}_{T_2}$. Consider the map
$$
\theta\co \ker(D) \to \ker(f)\oplus \ker(g)
$$
given by $\theta((t,r)) = (t - f^*(h(r)), r)$. It is easy to see that this map is a graded isomorphism.
\end{proof}

Let $\hvti$ be the restriction of $\hv$ to $\hAr$ followed by the projection to $\hBr$. Define $\hhti$ similarly using $\hh$. Define also $\hDti$ to be the restriction of $\hD$ to $\hAAr$ followed by the projection to $\hBBr$. The map $\hDti$ is a sum of various maps $\hvti$ and $\hhti$.

In terms of notation in Lemma \ref{lemma-vec-sp} we have $D = \hD$, $f = \hDT$, $g = \hDti$, $T_1 = \hAAT$, $T_2 = \hBBT$, $R_1 = \hAAr$ and $R_2 = \hBBr$.

\begin{lemma}
Let $Y$ be a homology sphere and $K \subset Y$ a knot with $V_0 = 0$. Define $Z = Y_{p/q}(K)$. Then $\dim(\ker(\hDti)), \dim(\mathrm{coker}(\hDti)) < \infty$,

$$
\dim(\ker(\hDti))+\dim(\mathrm{coker}(\hDti))\leq \dim(HF_{red}(Z,i))+\dim(HF_{red}(Y))
$$

\

and

$$
\dim(\ker(\hDti)) \leq \dim(HF_{red}(Z,i)).
$$
\label{lemma-ker-coker-ineq}
\end{lemma}
\begin{proof}
By the proof of Proposition \ref{prop-d-ineq} (specifically inequality \eqref{d-ineq-lower}) the tower in $HF^+(Z,i)$ is isomorphic to a direct sum of two pieces (one of which may be trivial). One piece is the kernel of $\hDT$ (the whole kernel, because $V_0 = 0$ -- see e.g. \cite[Lemma 13]{gainullin}). Another piece (which may be trivial) is a subspace of the co-kernel of $\hD$ (isomorphic to the co-kernel of $\hDti$) of dimension at most $\dim(HF_{red}(Y))$. 

Now by Lemma \ref{lemma-vec-sp} the whole resulting homology is isomorphic (in a graded way) to the kernel of $\hDT$ and the homology of $\hDti$ (which correspond to $f$ and $g$ in Lemma \ref{lemma-vec-sp} respectively). The paragraph above implies that the whole tower part is covered by the kernel of $\hDT$ and perhaps a piece of dimension at most $\dim(HF_{red}(Y))$. Thus the homology of $\hDti$ is finite dimensional\footnote{The structure of the graded vector spaces here allows us to talk about the dimensions: once we identified what covers the elements of high enough grading in the tower, the rest has to happen in a finite-dimensional vector space.}, so $\dim(\ker(\hDti)), \dim(\mathrm{coker}(\hDti)) < \infty$ and since the dimension of the homology of $\hDti$ equals $\dim(\ker(\hDti))+\dim(\mathrm{coker}(\hDti))$ we have

$$
\dim(\ker(\hDti))+\dim(\mathrm{coker}(\hDti))\leq \dim(HF_{red}(Z,i))+\dim(HF_{red}(Y))
$$

For the second inequality note that the reduced part of $HF^+(Z,i)$ consists of the part in the kernel and the part in the cokernel. If we forget about the part in the cokernel altogether, we can see that the kernel contributes to the dimension exactly $\dim(\ker(\hDti))$. This verifies the second inequality.
\end{proof}

For an absolutely $\ZZ_2$-graded abelian group $H$ let $H_e$ denote the subgroup of elements of grading $0$ and $H_o$ denote the subgroup of elements of grading $1$. 

Recall that we defined
$$
N(Y,Z) = 2|H_1(Z)|\dim(HF_{red}(Y))+\dim(HF_{red}(Z)).
$$

We are now ready to prove the main theorem of this section.

\begin{theorem}
Let $Y$ be a non-$L$-space homology sphere, $Z$ be a rational homology sphere and $K \subset Y$ be a knot and suppose there are coprime integers $p,q$ such that $Z = Y_{p/q}(K)$.

If $|q|>N(Y,Z)$, then

\begin{itemize}
\item $V_0(K) = 0$;
\item $\Delta_K \equiv 1$;
\item $\dim(\hAr_e) = \dim(HF_{red}(Y)_{e})$ for all $k$;
\item $\dim(\hAr_o) = \dim(HF_{red}(Y)_{o})$ for all $k$.
\end{itemize}
\label{theorem-K-special}
\end{theorem}

\begin{proof}
Suppose $|q|>N(Y,Z)$. We know from Lemma \ref{lemma-v0-not-0} that $V_0(K) = 0$. Since changing the orientation of a manifold does not change the dimension of its reduced Floer homology \cite[Proposition 2.5]{OSzPropApp}, we can assume that $p>0$ and $q>0$.

\

\begin{claim}
$\dim(\hAr_{e}) \geq \dim(\hBr_{e})$ and $\dim(\hAr_{o}) \geq \dim(\hBr_{o})$.
\label{claim1}
\end{claim}

\

\begin{claimproof}
Even and odd cases are completely analogous so we only prove $\dim(\hAr_{e}) \geq \dim(\hBr_{e})$. Suppose for contradiction that $\dim(\hBr_{e})-\dim(\hAr_{e}) \geq 1$.

Let $B_i$ be the sum of all $(n, \hBr_{e})$ with $n$ satisfying $\floor{\frac{i+pn}{q}}=\floor{\frac{i+p(n-1)}{q}}=k$. Let $A_i$ be the sum of all $(n, \hAr_{e})$ with $n$ satisfying $\floor{\frac{i+pn}{q}}=k$. Then $\hDti$ maps $A_i$ into $B_i$, so $\dim(\mathrm{coker}(\hDti))\geq \dim(B_i)-\dim(A_i)$.

Define $N_i = \{\ j \ |\ j\equiv i \ (\mbox{mod }p), \floor{\frac{j}{q}} = k\}$. Then $\dim(A_i) = N_i\dim(\hAr_{e})$ and $\dim(B_i) = (N_i-1)\dim(\hBr_{e})$.

By Lemma \ref{lemma-ker-coker-ineq} we have
\begin{multline*}
\dim(HF_{red}(Z,i))+\dim(HF_{red}(Y)) \geq \dim(\mathrm{coker}(\hDti))\geq \\
\geq N_i(\dim(\hAr_{e})-\dim(\hBr_{e}))-\dim(\hBr_{e})\geq N_i - \dim(HF_{red}(Y)).
\end{multline*}

Noting that $\sum_{i=0}^{p-1}N_i = q$ and summing over all \Sps\ we get
$$
\dim(HF_{red}(Z))+p\dim(HF_{red}(Y))\geq q - p\dim(HF_{red}(Y)),
$$

which contradicts the assumption made on $q$.
\end{claimproof}

\

\begin{claim}
$\dim(\hAr_{e}) \leq \dim(\hBr_{e})$ and $\dim(\hAr_{o}) \leq \dim(\hBr_{o})$.
\label{claim2}
\end{claim}

\

\begin{claimproof}
Again, the two cases are analogous, so we only show that $\dim(\hAr_{e}) \leq \dim(\hBr_{e})$. Suppose for a contradiction that $\dim(\hAr_{e}) - \dim(\hBr_{e}) \geq 1$.

Let $\widehat{A}_i$ be the sum of all $(n, \hAr_{e})$ with $n$ satisfying $\floor{\frac{i+pn}{q}}=\floor{\frac{i+p(n+1)}{q}}=k$. Let $\widehat{B}_i$ be the sum of all $((n, \hBr_{e})$ with $n$ satisfying $\floor{\frac{i+pn}{q}}=k$.

Clearly $\hDti$ maps $\widehat{A}_i$ into $\widehat{B}_i$, so $\dim(\ker(\hDti))\geq \dim(\widehat{A}_i)-\dim(\widehat{B}_i))$.

We have $\dim(\widehat{B}_i) = N_i\dim(\hBr_{e})$ and $\dim(\widehat{A}_i) = (N_i-1)\dim(\hAr_{e})$.

Hence by Lemma \ref{lemma-ker-coker-ineq} we have
\begin{multline*}
\dim(HF_{red}(Z,i)) \geq \dim(\ker(\hDti)) \geq \\
\geq (N_i-1)(\dim(\hAr_{e})-\dim(\hBr_{e}))-\dim(\hBr_{e}) \geq \\
\geq N_i-1 - \dim(HF_{red}(Y))
\end{multline*}

Summing over all \Sps\ we get
$$
\dim(HF_{red}(Z))\geq q-p-p\dim(HF_{red}(Y)),
$$
which is again a contradiction to the assumed inequality for $q$.
\end{claimproof}

\

Combining the results in the two Claims we see that the assumption that $q$ violates the bound in the statement of the Lemma implies that for all $k$ we have $\dim(\hAr_{e}) = \dim(\hBr_{e})$ and $\dim(\hAr_{o}) = \dim(\hBr_{o})$. Thus $\chi(\hAr)=\chi(\hBr)$ for all $k$.

Let $K$ be a knot and $\Delta_K(T) = a_0 + \sum_i a_i (T^i+T^{-i})$ be its symmetrised Alexander polynomial, with normalisation convention $\Delta_K(1) = 1$. Define its \slshape torsion coefficients \upshape $t_i(K)$ for $i\geq 0$ by
\begin{equation*}
t_i(K) = \sum_{j\geq 1} ja_{i+j}.
\end{equation*}

We now want to show that $t_k(K) = \chi(\hAr)-\chi(\hBr)$. This will imply that all torsion coefficients of $K$ are $0$ and thus its Alexander polynomial is trivial.

Define $\Delta_k = C\{j \geq k \mbox{ and } i < 0\}$. Note that $\chi(\Delta_k) = t_k(K)$. We have an exact sequence
\begin{center}
{\setlength\mathsurround{0pt}
\begin{equation*}
\begin{tikzcd}
0 \arrow{r}{} &\Delta_k \arrow{r}{i} &\hA \arrow{r}{\cv} &\hB \arrow{r}{} &0,
\end{tikzcd}
\end{equation*}
}
\end{center}

which leads to an exact triangle

{\setlength\mathsurround{0pt}
\begin{center}
\begin{equation}
\begin{tikzcd}
H_*(\Delta_k) \arrow{r}{i_*}\arrow[leftarrow]{rd}
&\hA \arrow{d}{\hv}\\
&\hB.
\end{tikzcd}
\label{triang-full}
\end{equation}
\end{center}
}

Since $V_0 = 0$, the map $\hvt$ maps $\hAT$ isomorphically onto $\hBT$. So, up to graded isomorphism, we also have an exact triangle

{\setlength\mathsurround{0pt}
\begin{center}
\begin{equation}
\begin{tikzcd}
H_*(\Delta_k) \arrow{r}{i_*}\arrow[leftarrow]{rd}
&\hAr \arrow{d}{\hvti}\\
&\hBr.
\end{tikzcd}
\label{triang-reduced}
\end{equation}
\end{center}
}

It follows that $t_k(K) = \chi(\Delta_k) = \chi(\hAr) - \chi(\hBr) = 0$.
\end{proof}

\section{A bound on $q$ for exceptional knots of genus larger than $1$}
\label{sec:K-spec-genus-1}

In this section we want to show that for knots that do not satisfy the bound of Theorem \ref{theorem-K-special} and have genus larger than $1$, $q$ is still bounded by a quantity depending only on the pair of manifolds connected by surgery. We first prove the following lemma, a version of which for $S^3$ was proven in \cite[Lemma 2.5]{homLidmanZufeltReducible}.

\begin{lemma}
Let $\{V_k\}_{k\in \ZZ}$ and $\{H_k\}_{k\in \ZZ}$ be numbers associated with a knot $K$ in a homology sphere, as defined in Section \ref{sec:mapCone}. Then
$$
H_k-V_k = k
$$
for all $k \in \ZZ$.
\label{lemma-v-h}
\end{lemma}
\begin{proof}
According to \cite[Theorem 2.3]{OSzIntegerSurgeries}, the modules $\hA$ can be identified with $HF^+$ of $N$-surgeries on $K$ (in a certain \Sp), where $N$ is a sufficiently large integer. Moreover, after this identification the maps $\hv$ and $\hh$ coincide with the maps into $HF^+(Y)$ induced by the (turned-around) surgery cobordism.

More specifically, the maps $\hv$ and $\hh$ can be thought of as the maps corresponding to the \Sps\ $\mathfrak{v}_k$ and $\mathfrak{h}_k$ respectively, where

$$
\langle c_1(\mathfrak{v}_k),[\widehat{F}] \rangle + N = 2k
$$

and

$$
\langle c_1(\mathfrak{h}_k),[\widehat{F}] \rangle - N = 2k.
$$

Here $[\widehat{F}]$ is the homology class of the surface obtained by capping off a Seifert surface $F$ of $K$ with the core of the 2-handle.

From this we can deduce that $c_1(\mathfrak{v}_k)^2 = -\frac{1}{N}(2k-N)^2$ and $c_1(\mathfrak{h}_k)^2 = -\frac{1}{N}(2k+N)^2$ (see \cite[Proposition 2.69]{manion} for a nice exposition of this calculation). The difference in the grading shifts of the two maps identified with $\hv$ and $\hh$ is given by $2(H_k - V_k)$. On the other hand, we can deduce from \cite[Theorem 7.1]{OSzHoloTriang} that the difference in the grading shifts is also given by
$$
\frac{c_1(\mathfrak{v}_k)^2-c_1(\mathfrak{h}_k)^2}{4} = 2k.
$$

Comparing the two expressions we get the desired result.
\end{proof}

For two homogeneous elements $u$, $v$ in the mapping cone complex, denote their relative $\ZZ$-grading by $\mathrm{deg}(u,v)$. For a homogeneous element $w$ of Heegaard Floer homology of some rational homology sphere, denote by $\mathrm{deg}(w)$ its absolute $\QQ$-grading. Recall that the modules $\hA$ and $\hB$ decompose as the sum of the `tower' $\Ta$ and the reduced part. For a homogeneous element $c$ in either one of $\hA$ or $\hB$ denote by $\widetilde{d}(c)$ its relative grading with the $1$ in the tower part (i.e. $\widetilde{d}(c) = \mathrm{deg}(c,1)$).

As already mentioned, if $Y$ is an $L$-space homology sphere, there is a bound on $q$ similar to that of Theorem \ref{theorem-K-special} that holds for all knots. Thus, as before, we will assume throughout this section that $Y$ is a non-$L$-space homology sphere.

\begin{lemma}
Let $K \subset Y$ be a knot and suppose $Z = Y_{p/q}(K)$, where $p,q>0$. Let $N(Y,Z)$ be defined as in the statement of Theorem \ref{theorem-K-special} and suppose $q>N(Y,Z)$. Then for every homogeneous $z\in \hAr$
$$
\widetilde{d}(z)\geq \min\{\widetilde{d}(c)\ |\ \mbox{homogeneous }c\in \hBr\}.
$$
\label{lemma-d-tilde}
\end{lemma}

\begin{proof}
Suppose there exists $k$ and $z\in \hAr$ with 

$$
\widetilde{d}(z)< \min\{\widetilde{d}(c)\ |\ \mbox{homogeneous }c\in \hBr\}.
$$

Both $\hv$ and $\hh$ do not increase $\widetilde{d}$, so $z$ is in the kernel of both $\hvti$ and $\hhti$, hence also in the kernel of $\hDti$. This holds for every copy of $\hAr$ in the mapping cone complexes for all \Sps, so summing contributions from all \Sps\ and using Lemma \ref{lemma-ker-coker-ineq} we deduce
$$
q\leq \dim(HF_{red}(Z))<N(Y,Z).
$$
This is a contradiction.
\end{proof}

Recall that for a rational homology sphere $Z$ we defined

$$
\widehat{D}(Z) = \max\{\mathrm{deg}(z)-d(Z,\mathfrak{t})\ |\ \text{homogeneous } z\in HF_{red}(Z,\mathfrak{t}),\mathfrak{t}\in \mathrm{Spin}^c(Z)\}
$$

and

$$
\widecheck{D}(Z) = \min\{\mathrm{deg}(z)-d(Z,\mathfrak{t})\ |\ \text{homogeneous } z\in HF_{red}(Z,\mathfrak{t}),\mathfrak{t}\in \mathrm{Spin}^c(Z)\},
$$

where $\mathrm{deg}(z)$ is the absolute grading of $z$.

We can now formulate a bound on $q$ that holds for knots that have genus $>1$ and are not covered by Theorem \ref{theorem-K-special}.

\begin{prop}
Let $K\subset Y$ be a knot such that $Z = Y_{p/q}(K)$ for $p,q>0$ and $q>N(Y,Z)$. Suppose the genus of $K$ is larger than one. Then
$$
\floor{q/p} \leq \frac{\widehat{D}(Z)-\widecheck{D}(Y)}{2}.
$$
\label{prop-excep-K-genus}
\end{prop}

\begin{proof}
By Theorem \ref{theorem-genus_detect} and the exact triangle \eqref{triang-full} the map $\boldsymbol{v}^+_{g-1}$ must not be an isomorphism (where $g$ is the genus of $K$). Since $V_0 = 0$, the map $\boldsymbol{v}^T_{g-1}$ is an isomorphism, so the map $\boldsymbol{\widetilde{v}}_{g-1}$ must not be an isomorphism. Since the spaces $\boldsymbol{A}^{red}_{g-1}(K)$ and $\boldsymbol{B}^{red}$ have the same dimension the map $\boldsymbol{\widetilde{v}}_{g-1}$ must have some kernel. Suppose $z \in \ker(\boldsymbol{\widetilde{v}}_{g-1})$. By adding an element of $\boldsymbol{A}^T_{g-1}(K)$ if necessary, we can assume that $z\in \ker(\boldsymbol{v}^+_{g-1})$.

Let $N = \max\{n\ |\ \floor{\frac{pn}{q}} = g-1\}$. Then $(N,z) \in (N, \boldsymbol{A}^+_{g-1}(K))$ is in the kernel of $\boldsymbol{v}^+_{g-1}$. By Lemma \ref{lemma-iso-at-ends} we can assume that $(N,z)$ is in the kernel of $\boldsymbol{D}^+_{0,p/q}$. Denoting as usual by $1$ the generators of the kernel of $U$ in the tower modules, we have
$$
\mathrm{deg}((N,1),(0,1)) = 2\sum_{i = 1}^{g-1}n_iH_i,
$$
where $n_i \geq \floor{q/p}$. Since $V_0 = 0$, by Lemma \ref{lemma-v-h} we have $H_i = i$ for $i\geq 0$. So
$$
\mathrm{deg}((N,1),(0,1)) \geq g(g-1)\floor{q/p}.
$$

By the proof of Proposition \ref{prop-d-ineq} $\mathrm{gr}(0,1)\geq d(Z,0)$. In homology $(N,z)$ represents some element of $HF_{red}(Z,0)$. Suppose its absolute grading there is $G$.

Then $G - d(Z,0) \geq \mathrm{deg}((N,z), (0,1)) = \widetilde{d}(z) + \mathrm{deg}((N,1),(0,1))$. By Lemma \ref{lemma-d-tilde}
$$
\widetilde{d}(z) \geq \widecheck{D}(Y).
$$

Combining the various inequalities we obtain
$$
\floor{q/p}\leq \frac{G-d(Z,0)-\widecheck{D}(Y)}{g(g-1)} \leq \frac{\widehat{D}(Z)-\widecheck{D}(Y)}{2}.
$$
\end{proof}

We can combine Theorem \ref{theorem-K-special} and Proposition \ref{prop-excep-K-genus} into the following

\begin{cor}
Let $Y$ be a homology sphere and $K\subset Y$ a knot. Suppose $Y_{p/q}(K)=Z$ for $p\neq 0$. There exists a constant $C(Y,Z)$ that depends only on the Heegaard Floer homology of $Y$ and $Z$ such that if $q>C(Y,Z)$ then
\begin{itemize}
\item the genus of $K$ is $1$;
\item $K$ has trivial Alexander polynomial;
\item $V_0(K) = 0$;
\item $\dim(\hAr_e) = \dim(HF_{red}(Y)_{e})$ for all $k$;
\item $\dim(\hAr_o) = \dim(HF_{red}(Y)_{o})$ for all $k$.
\end{itemize}
\end{cor}

\section{Some knots determined by their complements}

Results of Section \ref{sec:Z-special} can be applied to show that in certain homology $\RR P^3$'s non-null-homologous knots are determined by their oriented complements.

\begin{cor}
Let $Z$ be a closed connected oriented manifold with $|H_1(Z)| = 2$. Suppose that $\dim(HF_{red}(Z))$ is odd. Then non-null-homologous knots in $Z$ are determined by their complements.
\label{cor-z-special-complement}
\end{cor}

\begin{proof}
By Theorem \ref{theorem-Z-special}, if a knot in one of these spaces has a homology sphere surgery, then it is determined by its complement. Thus it will be enough to show that non-null-homologous knots in such spaces have homology sphere surgeries.

Let $Z$ be a space as in the statement. Denote by $S$ a solid torus regular neighbourhood of $K$ and denote the exterior of $K$ by $Z_0$. Let $T$ be the boundary of $S$. Consider the exact sequence for the pair $(Z, S)$.

{\setlength\mathsurround{0pt}
\begin{center}
\begin{tikzcd}
0 \arrow{r}
&H_2(Z,S) \arrow{r}
&H_1(S) \arrow{r}
&H_1(Z) \arrow{r}
&H_1(Z,S) \arrow{r}
&0
\end{tikzcd}
\end{center}
}
By considering this sequence with coefficients in $\QQ$, $H_2(Z,S)$ is a direct sum of one copy of $\ZZ$ with a torsion group. Then excision and Poincar\'e--Lefschetz duality show that $H_2(Z,S) \cong H_2(Z_0,T) \cong H^1(Z_0)$. The latter group is free abelian, thus $H_2(Z_0,T) \cong \ZZ$. Similarly $H_1(Z,S)\cong H^2(Z_0)$, which is equal (by the Universal Coefficients Theorem) to the torsion part of $H_1(Z_0)$, call it $G$. We have $H_1(Z_0) \cong \ZZ \oplus G$. Now the sequence above becomes

{\setlength\mathsurround{0pt}
\begin{center}
\begin{tikzcd}
0 \arrow{r}
&\ZZ \arrow{r}
&\ZZ \arrow{r}
&\ZZ_2 \arrow{r}
&G \arrow{r}
&0.
\end{tikzcd}
\end{center}
}

Since $K$ is not null-homologous (but obviously $2K$ is), there is a rational Seifert surface for $K$ that winds twice longitudinally. This implies that the map between two copies of $\ZZ$ in the sequence above is multiplication by $2$. It follows that $G = 0$.

Now consider the exact sequence of the pair $(T,Z_0)$:

{\setlength\mathsurround{0pt}
\begin{center}
\begin{tikzcd}
\ \arrow{r}
&H_2(Z_0,T) \arrow{r}
&H_1(T) \arrow{r}
&H_1(Z_0) \arrow{r}
&0.
\end{tikzcd}
\end{center}
}

This sequence shows that $H_1(Z_0)$ is generated by the images of the meridian and the longitude. This means that if we perform a surgery with a slope given by the generator of $H_1(Z_0)$, we get a homology sphere.
\end{proof}

The Brieskorn sphere $\Sigma(2,3,7)$ has perhaps the simplest Heegaard Floer homology a non-$L$-space can possibly have -- the rank of its reduced Floer homology is $1$. This makes it possible to show that `most' knots in $\Sigma(2,3,7)$ are determined by their complements. The first part of the proof shows that a surgery from $\Sigma(2,3,7)$ to $\Sigma(2,3,7)$ must be integral. The thinking behind the proof is similar to that of Section \ref{sec:K-special}, but we get a better bound due to the fact that $1<2$ and linear maps into 1-dimensional spaces are either trivial or surjective.

\begin{theorem}
Knots of genus larger than $1$ in the Brieskorn sphere $\Sigma(2,3,7)$ are determined by their complements. Moreover, if $K\subset \Sigma(2,3,7)$ is a counterexample to Conjecture \ref{conj-complement} then the surgery slope is integral, $\widehat{HFK}(\Sigma(2,3,7),K,1)$ has dimension $2$ and its generators lie in different $\ZZ_2$-gradings.

Non-fibred knots of genus larger than $1$ in $\Sigma(2,3,7)$ are strongly determined by their complements.
\label{theorem-brieskorn}
\end{theorem}

\begin{proof}
The Heegaard Floer homology of $-\Sigma(2,3,7)$ has been computed in \cite[Section 8.1]{OSzAbsGr}. Alternatively, we can calculate it using the program HFNem2 by \c{C}a\u{g}r{\i} Karakurt\footnote{At the time of writing available for download at \url{https://www.ma.utexas.edu/users/karakurt/}}. It is clear that proving what we want for knots in $-\Sigma(2,3,7)$ is equivalent to proving it for knots in $\Sigma(2,3,7)$. 

The calculation shows that $-\Sigma(2,3,7)$ has reduced Floer homology of rank $1$, situated in the absolute grading $0$ which is equal to the $d$-invariant of $-\Sigma(2,3,7)$. By \cite[Proposition 2.5]{OSzPropApp} and \cite[Proposition 4.2]{OSzAbsGr} reduced Floer homology of $\Sigma(2,3,7)$ has rank one (and odd absolute $\ZZ_2$-grading) and $d$-invariant $0$.

Suppose a knot $K\subset \Sigma(2,3,7)$ gives $\Sigma(2,3,7)$ by $1/q$-surgery. Since the analogous statement can be proven for $-\Sigma(2,3,7)$, we can assume that $q>0$.

Denote $Y = \Sigma(2,3,7)$. Suppose for a contradiction that $q>1$.  Note that by Proposition \ref{prop-d-ineq} we must have $V_0 = 0$. Consider $(qg-1,\boldsymbol{A}^+_{g-1})$, where $g = g(K)$. This is the `rightmost' group for which the corresponding map $\boldsymbol{v}$ is not an isomorphism.

Suppose $\boldsymbol{A}^{red}_{g-1} = 0$. Since $q\geq 2$, $(qg-1,\boldsymbol{B}^{red})$ is not in the image of $\hD$. Thus it gives rise to a generator of $HF_{red}(Y)$. However, this element is in the even $\ZZ_2$-grading. This gives a contradiction.

Now assume $\dim(\boldsymbol{A}^{red}_{g-1}) \geq 1$. Since the map $\boldsymbol{v}_{g-1}$ cannot be an isomorphism, the map $\boldsymbol{v}^{red}_{g-1}$ must have kernel. According to Lemma \ref{lemma-iso-at-ends} this element of the kernel gives rise to an element in $HF_{red}(Y)$. However, applying the same argument `on the other end', i.e. to $\boldsymbol{h}_{-(g-1)}$, we see that we must have $\dim(HF_{red}(Y))\geq 2$, which gives a contradiction.

All in all, we have $q=1$, irrespective of the genus. Now assume $g>1$. According to \cite[Theorem 9.1]{OSzPropApp}, there is an exact triangle of $\FR$-modules

{\setlength\mathsurround{0pt}
\begin{center}
\begin{equation}
\begin{tikzcd}
HF^+(Y) \arrow{r}{f}\arrow[leftarrow]{rd}{h}
&HF^+(Y_0(K)) \arrow{d}{g}\\
&HF^+(Y_1(K)).
\label{fig:extriang-surgeries}
\end{tikzcd}
\end{equation}
\end{center}
}

Much as in the proof of \cite[Corollary 4.5]{OSzKnotInv} we have\footnote{The labellings of the \Sps\ are obtained by taking the half of the pairing of their Chern classes with the generator of $H^2(Y_0(K))$ obtained by capping a Seifert surface.} 
$$
HF^+(Y_0(K),g-1) \cong HF^+(Y_0(K),-(g-1)) \cong \widehat{HFK}(Y,K,g),
$$

so these groups are non-trivial by Theorem \ref{theorem-genus_detect}.

By \cite[Theorem 10.4]{OSzPropApp} $HF^+(Y_0(K),0)\cong \mathcal{T}_{d_{-1/2}(Y_0(K))}\oplus \mathcal{T}_{d_{+1/2}(Y_0(K))}\oplus R$, where $R = HF_{red}(Y_0(K),0)$ is a finite-dimensional vector space in the kernel of some power of $U$ and $d_{\pm 1/2}(Y_0(K)) \in \ZZ + \frac{1}{2}$ is an analogue of the correction term.

By \cite[Lemma 3.1]{OSzAbsGr} the component of $f$ mapping into $HF^+(Y_0(K),0)$ has grading $-1/2$ and the restriction of $g$ to $HF^+(Y_0(K),0)$ also has grading $-1/2$.

Since in high enough gradings the group $HF^+(Y_0(K), 0)$ consists of two tower modules (in different relative gradings) and only one of them can be in the image of the map $f$, it follows that the other one has to map onto the tower in $HF^+(Y_1(K))$. It follows that the restriction of $h$ to the tower part of $HF^+(Y_1(K))$ is zero. Since we assume $g>1$ and so $HF^+(Y_0(K))$ contains two non-trivial vector spaces $HF^+(Y_0(K),\pm(g-1))$, $h$ has to be trivial. If $h$ was not trivial, it would surject onto the reduced part of $HF^+(Y)$ and so there would be nothing left to map onto $HF^+(Y_0(K),\pm(g-1))$. Thus we can assume $h = 0$. Then the maps $f$ and $g$ map a tower module isomorphically onto another one, so comparing the gradings we see that $d_{\pm 1/2}(Y_0(K)) = \pm 1/2$.

Moreover, the dimension of $HF_{red}(Y_0(K))$ has to be 2, thus
$$
HF^+(Y_0(K),g-1) \cong HF^+(Y_0(K),-(g-1)) \cong \widehat{HFK}(Y,K,g) \cong \FF_2,
$$

$$
HF^+(Y_0(K),k)  = 0
$$

for $k \not \in \{0,\pm(g-1)\}$ and $R = 0$.

By \cite[Proposition 10.14 and Theorem 10.17]{OSzPropApp} we have
$$
t_k(K) = 0
$$

for $0\leq k <g$ and $t_{g-1}(K) = \pm 1$.

Notice that
$$
\Delta_K''(1) = 2t_0(K) + 4\sum_{i=1}^{g-1}t_i(K) = \pm 4 \neq 0,
$$

which contradicts equation \eqref{lambda-surg}.

If $g = q = 1$, then by the reasoning of Lemma \ref{lemma-iso-at-ends} $\boldsymbol{A}^{red}_0\cong HF_{red}(Y) = \FF_2$. Since $\boldsymbol{v}^{red}_0$ cannot be an isomorphism, it must be zero. It follows that the dimension of $\widehat{HFK}(Y,K,1)$ is $2$. Since $\Delta_K''(1) = 0$ forces the Alexander polynomial to be trivial, the two generators have to be in different $\ZZ_2$-gradings.

Now suppose $K\subset Y$ produces $-Y$ by $1/q$-surgery for $q>0$. Moreover, let $g(K)>1$. Just as before we cannot have $\dim(\boldsymbol{A}^{red}_{g-1}) \geq 1$, thus $\boldsymbol{A}^{red}_{g-1} = 0$. Since the $d$-invariants of $Y$ and $-Y$ coincide, we have $V_0 = 0$, thus the map $\boldsymbol{v}^{red}_{g-1}$ has no kernel and has co-kernel of dimension $1$. It follows that $\dim(\widehat{HFK}(Y,K,g(K))) = 1$, so by \cite[Theorem 1.1]{niDetectsFibred} or \cite[Theorem 9.11]{juhaszSurfaceDecompositions} $K$ is fibred.
\end{proof}

\section{Null-homologous knots in $L$-spaces}

The aim of this section is to show that null-homologous knots in $L$-spaces are determined by their complements. Moreover, we will also show that all knots in lens spaces of the form $L(p,q)$ with $p$ square-free are determined by their complements.

First of all, we need to verify that the mapping cone formula for rational surgeries as proved in \cite{OSzRatSurg} for knots in integer homology spheres also applies to null-homologous knots in rational homology spheres with only minor modifications.

Let $Y$ be a rational homology sphere and $K$ a null-homologous knot in it. Let $Y_0$ be the exterior of $K$. To understand relative \Sps\ we first need to calculate the first homology.

\begin{lemma}
We have $H_1(Y_0) \cong \ZZ \oplus H_1(Y)$, where the first factor is the group generated by the meridian of $K$.
\label{lemma-long-exact}
\end{lemma}

\begin{proof}
Let $S$ be a (closed) regular neighbourhood of $K$ and $T = \partial Y_0$ its boundary. Consider the long exact sequence for the pair

{\setlength\mathsurround{0pt}
\begin{center}
\begin{equation*}
\begin{tikzcd}
0 \arrow{r}{} &H_2(Y,S) \arrow{r}{} &H_1(S) \arrow{r}{} &H_1(Y) \arrow{r} &H_1(Y,S) \arrow{r} &0.
\end{tikzcd}
\end{equation*}
\end{center}
}

By Excision, $H_*(Y,S) \cong H_*(Y_0,T)$. Since $K$ is null-homologous, the boundary of its Seifert surface generates $H_1(S)$. Thus the map $H_2(Y,S) \to H_1(S)$ is an isomorphism.

It follows that $H_1(Y) \cong H_1(Y_0,T)$. By Poincar\'e--Lefschetz duality $H_1(Y_0,T) \cong H^2(Y_0)$ and by the Universal Coefficients Theorem the torsion part of $H_1(Y_0)$ is equal to the torsion part of $H^2(Y_0)$, which, by the previous sentence, equals $H_1(Y)$.

Now consider the Mayer-Vietoris sequence

{\setlength\mathsurround{0pt}
\begin{center}
\begin{equation*}
\begin{tikzcd}
0 \arrow{r} &H_1(T) \arrow{r} &H_1(S)\oplus H_1(Y_0) \arrow{r} & H_1(Y) \arrow{r} &0.
\end{tikzcd}
\end{equation*}
\end{center}
}

The same sequence with rational coefficients shows that $H_1(Y_0)$ has rank one, so combining with the argument above, $H_1(Y_0) \cong \ZZ \oplus H_1(Y)$. Moreover, the map $H_1(T) \to H_1(S)$ takes the longitude to a generator and the meridian to $0$ and also the longitude is trivial in $H_1(Y_0)$. It follows that we have an exact sequence

{\setlength\mathsurround{0pt}
\begin{center}
\begin{equation*}
\begin{tikzcd}
0 \arrow{r} &\langle m \rangle=\ZZ \arrow{r}{f} &H_1(Y_0)\cong \ZZ\oplus H_1(Y) \arrow{r} & H_1(Y) \arrow{r} &0.
\end{tikzcd}
\end{equation*}
\end{center}
}

Since the map $f$ is an injection, the subgroup $0 \oplus H_1(Y)$ is disjoint from the image of $f$. Hence it is mapped injectively into $H_1(Y)$. However, since this group is finite, the restriction of the second map from the right to $H_1(Y)$ is an isomorphism. It follows that the $\ZZ$ factor of $H_1(Y_0)$ (perhaps after changing the splitting) is equal to the kernel of this map and thus is generated by the meridian.
\end{proof}

Since $H_1(Y_0)$ acts freely and transitively on the set of relative \Sps\ on  $Y_0$, we can identify (non-canonically) this group with the set of relative \Sps. In other words, we label relative \Sps\ on $Y_0$ by a pair $(n,h)$ with $n\in \ZZ$ and $h\in H_1(Y)$. Moreover, by adding a multiple of the meridian we can ensure that
$$
\langle c_1((n,h)),[F] \rangle = 2n,
$$

where $F$ is a Seifert surface for $K$.

Recall from \cite[Section 7]{OSzRatSurg} that doing $p/q$ surgery on $K$ in $Y$ is equivalent to doing integral surgery with slope $a = \floor{p/q}$ on the knot $\widehat{K} = K\#O_{q/r}$ in $\widehat{Y} = Y\#(-L(q,r))$. Here $p = aq + r$ and $O_{q/r}$ is a core of one of the Heegaard solid tori in $-L(q,r)$, thought of as the image of one component of the Hopf link after the $-q/r$-surgery on the other component.

Denote by $\widehat{Y}_0$ and $L_0$ the exteriors of $\widehat{K}_0$ and $O_{q/r}$ respectively. Notice that $\widehat{Y}_0$ is obtained by glueing $Y_0$ and $L_0$ along an annulus $A$ whose core maps to meridians of $K$ and $O_{q/r}$. Consider the associated Mayer-Vietoris sequence

{\setlength\mathsurround{0pt}
\begin{center}
\begin{equation*}
\begin{tikzcd}
\arrow{r} &H_1(A) \arrow{r} &H_1(Y_0)\oplus H_1(L_0) \arrow{r} &H_1(\widehat{Y}_0) \arrow{r} &0.
\end{tikzcd}
\end{equation*}
\end{center}
}

We can rewrite it as

{\setlength\mathsurround{0pt}
\begin{center}
\begin{equation*}
\begin{tikzcd}
\arrow{r} &\ZZ \arrow{r} &\ZZ \oplus H_1(Y)\oplus \ZZ \arrow{r} &H_1(\widehat{Y}_0) \arrow{r} &0.
\end{tikzcd}
\end{equation*}
\end{center}
}

Moreover, the generator of $H_1(A)$ is mapped to meridians of both knots. Thus the second map from the right is given by $1 \mapsto (1,0,-q)$. It follows that we can write $H_1(\widehat{Y}_0) \cong \ZZ \oplus H_1(Y)$. Moreover, the meridian of $K$ maps to $q$ times the generator of $\ZZ$ and the meridian of $O_{q/r}$ maps to the generator of $\ZZ$. As in the case of homology spheres, the push-off of $K$ with respect to the framing $a$ is mapped to $p$ times the generator of the $\ZZ$ summand.

Now just as in \cite[Proof of Theorem 1.1]{OSzRatSurg} we can assemble the mapping cone, whose homology will coincide with the Heegaard Floer homology of $p/q$-surgery on $K$. The only difference with the case of homology spheres is that instead of indices $i$ (that represented relative \Sps) we have to use pairs $(i,h)$, but when considering every \Sp\ on the resulting space separately, $h$ stays the same.

We use the same notation as in \cite{OSzRatSurg}, replacing relative \Sps\ with the labelling we defined after the proof of Lemma \ref{lemma-long-exact}, i.e. $(n,h)$. Note that $B^+_{(n,h)}$ only depends on $h$ (up to a shift in the filtration) and in homology gives $HF^+(Y,h)$.\footnote{The statement should be read as saying that there is an identification between $H_1(Y)$ and $\mathrm{Spin}^c(Y)$ such that this correspondence holds.} Consequently we denote this group simply by $B^+_h$.

For each $h \in H_1(Y)$ and $0\leq i <p$ consider the set of groups $(s,A^+_{(\floor{(i+ps)/q},h)})$ for $s \in \ZZ$. Combine them into

$$
\mathcal{A}^+_{(i,h)} = \bigoplus_{s\in \ZZ} (s, A^+_{(\floor{(i+ps)/q},h)}).
$$

Similarly, define

$$
\mathcal{B}_{(i,h)} = \bigoplus_{s\in \ZZ} (s,B^+_h).
$$

Define $D^+_{(i,h),p/q}\co\mathcal{A}^+_{(i,h)} \to \mathcal{B}_{(i,h)}$ component-wise by

$$
D^+_{(i,h),p/q}(s,a_s) = (s,v^+_{(\floor{(i+ps)/q},h)}(a_s))+(s+1,h^+_{(\floor{(i+ps)/q},h)}(a_s)).
$$

Denote by $\mathbb{X}^+_{(i,h),p/q}$ the mapping cone of $D^+_{(i,h),p/q}$.Then the Heegaard Floer homology of $Y_{p/q}(K)$ in a certain \Sp\ is given by the homology of $\mathbb{X}^+_{(i,h),p/q}$. We denote this \Sp\ on $Y_{p/q}(K)$ by $(i,h)$. In other words, we have

$$
H_*(\mathbb{X}^+_{(i,h),p/q}) \cong HF^+(Y_{p/q}(K),(i,h)).
$$

Reusing the notation from above, let $\boldsymbol{A}^+_{(n,h)}$, $\boldsymbol{B}^+_h$, $\mathbb{A}^+_{(i,h)}$ and $\mathbb{B}_{(i,h)}$ be the homologies of $A^+_{(n,h)}$, $B^+_h$, $\mathcal{A}^+_{(i,h)}$ and $\mathcal{B}_{(i,h)}$ respectively. Let $\boldsymbol{v}^+_{(n,h)}$, $\boldsymbol{h}^+_{(n,h)}$ and $\boldsymbol{D}^+_{(i,h),p/q}$ be the maps induced by $v^+_{(n,h)}$, $h^+_{(n,h)}$ and $D^+_{(i,h),p/q}$ respectively in homology.

As before, we denote by $\boldsymbol{D}^T_{(i,h),p/q}$, $\boldsymbol{v}^T_{(n,h)}$ and $\boldsymbol{h}^T_{(n,h)}$ the restrictions of $\boldsymbol{D}^+_{(i,h),p/q}$, $\boldsymbol{v}^+_{(n,h)}$ and $\boldsymbol{h}^+_{(n,h)}$ to the tower parts (in the case of $\boldsymbol{D}^T_{(i,h),p/q}$ the restriction to the sum of the tower parts).

The maps $\boldsymbol{v}^T_{(n,h)}$ and $\boldsymbol{h}^T_{(n,h)}$ are multiplications by powers of $U$. Denote these powers by $V_{(n,h)}$ and $H_{(n,h)}$ respectively. For each $h \in H_1(Y)$ we have:

\begin{itemize}
\item $V_{(n,h)} \geq V_{(n+1,h)}$ and $H_{(n,h)} \leq H_{(n+1,h)}$;
\item there is $N \in \mathbb{N}$ such that $V_{(n,h)} = 0$ for all $n \geq N$ and $H_{(n,h)} = 0$ for all $n \leq -N$;
\item $V_{(n,h)} \to + \infty$ as $n \to - \infty$ and $H_{(n,h)} \to + \infty$ as $n \to + \infty$.
\end{itemize}

By \cite[Lemma 12]{gainullin} the map $\boldsymbol{D}^T_{(i,h),p/q}$ is surjective (when $p/q>0$), so \linebreak $HF^+(Y_{p/q}(K),(i,h)) \cong \ker(\boldsymbol{D}^+_{(i,h),p/q})$.

If one of $\boldsymbol{A}^+_{(n,h)}$ is not a tower, i.e. contains some reduced part (which we denote by $\boldsymbol{A}^{red}_{(n,h)}$), then every element of $\boldsymbol{A}^{red}_{(n,h)}$ will be a component of some element of the kernel of $\boldsymbol{D}^+_{(i,h),p/q}$. However, such an element will not be in the image of large enough power of $U$. It follows that $HF^+(Y_{p/q}(K),(i,h))$ will have some reduced Floer homology. Thus if $Y_{p/q}(K)$ is an $L$-space, then $\boldsymbol{A}^+_{(n,h)} \cong \Ta$ for all $n$ and $h$.

Denote $\widehat{A}_{(n,h)} = \ker(U\co A^+_{(n,h)}\to A^+_{(n,h)})$ and its homology by $\widehat{\boldsymbol{A}}_{(n,h)}$. Since $\boldsymbol{A}^+_{(n,h)} \cong \Ta$ for all $n$ and $h$ we have $\widehat{\boldsymbol{A}}_{(n,h)} \cong \FF$ for all $n$ and $h$.

\subsection{Alexander polynomial}

Just as in the case of homology spheres, given a knot $K\subset Y$ in a rational homology sphere, one can define its Alexander module to be the first homology of the covering space $\widehat{Y}$ of $Y \setminus K$ with deck transformation group $\ZZ$. The two differences are as follows: firstly, to define $\widehat{Y}$ instead of the abelianisation map we use $\phi\co \pi_1(Y\setminus K)\to \ZZ$ gotten by composing abelianisation with the projection onto $\ZZ$ (so the subgroup defining $\widehat{Y}$ is the preimage of the torsion subgroup of $H_1(Y)$ under the abelianisation map). Secondly, for a more convenient definition of the Alexander polynomial later we use the ring $\QQ[t,t^{-1}]$ instead of $\ZZ[t,t^{-1}]$.

With these changes, the method for obtaining the presentation matrix for $H_1(\widehat{Y})$ as a $\QQ[t,t^{-1}]$-module using Fox calculus works in the same way as for knots in homology spheres -- see \cite[Chapter 11]{lickorishIntroBook}.

Now the Alexander polynomial $\Delta_K$ is defined to be a specific generator of the ideal generated by the maximal size minors of the presentation matrix of the Alexander module. The specific generator is fixed by the requirement that $\Delta_K(t) = \Delta_K(t^{-1})$ and $\Delta_K(1) = |H_1(Y)|$.

Suppose we have a genus $g$ doubly-pointed Heegaard diagram for $K\subset Y$. To get a Heegaard diagram of the knot exterior we add one more $\alpha$-curve, $\alpha_{g+1}$. As described in \cite[Section 3]{rasmussenThesis} this leads to a presentation of $\pi_1(Y)$ in which there is one generator for each $\alpha$-curve and one relator for each $\beta$-curve. Denote the generators by $\{a_i\}_{i=1}^{g+1}$ and relators (words in $a_i$) by $\{w_j\}_{j=1}^g$. Denote the free differential with respect to $a_i$ by $d_{a_i}$ (this time with respect to the map $\phi$, not abelianisation).

Define 

$$
\widehat{HFK}(Y,K,n) = \bigoplus_{h \in H_1(Y)} \widehat{HFK}(Y,K,(n,h)).
$$

Then as in \cite[Section 3]{rasmussenThesis} we can see that

$$
\chi(\widehat{HFK}(Y,K)) = \sum_{i,j} (-1)^it^j\dim\widehat{HFK}_i(Y,K,j) = \det(d_{a_i}w_j)_{1\leq i,j\leq g}.
$$

Following \cite[Proposition 3.1]{rasmussenLsp} we see that in fact

$$
\chi(\widehat{HFK}(Y,K)) = \Delta_K(t).
$$

For every $h \in H_1(Y)$\footnote{This definition works only up to an affine identification, since we should really have $h \in \mathrm{Spin}^c(Y)$.} define

$$
\Delta_{K,h}(t) = \sum_{i,j} (-1)^it^j\dim\widehat{HFK}_i(Y,K,(j,h)).
$$

We have $\Delta_K(t) = \sum_h \Delta_{K,h}(t)$ and $\Delta_{K,h}(1) = 1$.

Let $Y$ be an $L$-space and $K$ a null-homologous knot in it with $L$-space surgery. Recall that $\widehat{\boldsymbol{A}}_{(n,h)} \cong \FF$ for all $n$ and $h$. Using the same algebraic manipulations as in \cite[Section 3]{OSzLensSpaceSurg} we deduce that, for each fixed $h$, $\widehat{HFK}(Y,K,(n,h))$ has dimension $0$ or $1$, successive copies of $\FF$ are concentrated in different $\ZZ_2$ gradings and the first (and the last) copies of $\FF$ are concentrated in grading $0$.

It follows that, for each $h$, $\Delta_{K,h}''(1) \geq 0$ and equality is only possible if $\Delta_{K,h}(t) = 1 \mbox{ or } t$.

\subsection{Knots determined by their complements}

We are now ready to prove the surgery characterisation of the unknot for null-homologous knots in rational homology $L$-spaces.

\begin{theorem}
Let $Y$ be an $L$-space and $K\subset Y$ a null-homologous knot. Suppose that
$$
HF^+(Y_{p/q}(K)) \cong HF^+(Y \# L(p,q)).
$$

Then $K$ is the unknot.

In particular, null-homologous knots in $L$-spaces are determined by their complements.
\label{theorem-unknot-characterisation}
\end{theorem}

\begin{proof}
The Casson-Walker invariant is additive under connected sums and by \eqref{lambda-hf} determined by Heegaard Floer homology, thus we have

$$
\lambda(Y) + \lambda(L(p,q)) = \lambda(Y_{p/q}(K)) = \lambda(Y) + \lambda(L(p,q))+\frac{q}{2p|H_1(Y)|}\Delta_K''(1).
$$

It follows that $\Delta_K''(1) = 0$. Thus for each $h$ we must have $\Delta_{K,h}(t) = 1 \mbox{ or } t$. However, by symmetry if there is a multiple of $t$ in $\Delta_K(t)$ there must also be a multiple of $t^{-1}$. Hence $\Delta_{K,h}(t) = 1$ for all $h$. Note that it also means that 
$$
\widehat{HFK}(Y,K,(n,h)) = 0
$$

for all $n \neq 0$. Hence by \cite[Theorem 2.2]{niWuRatGenus} $g(K) = 0$, i.e. $K$ is the unknot.
\end{proof}

A straightforward homological argument provides more restrictions on knots not being determined by their complements in lens spaces. In particular, all knots in lens spaces $L(p,q)$ with $p$ square-free satisfy Conjecture \ref{conj-complement}. For notation used in the statement below see Section \ref{sec:intro} or the proof. By a non-trivial surgery we mean a surgery with a slope that is not the meridian, so even if a slope is equivalent to the meridian, surgery with this slope is still non-trivial.

It is clear that cores of Heegaard solid tori of $L(p,q)$ admit non-trivial surgery which give back $L(p,q)$. Hence, in the Corollary below we assume that $K$ is not a core of one of the Heegaard solid tori.

\begin{cor}
If $p$ is square-free, then all knots in $L = L(p,q)$ are determined by their complements.

More precisely, let $K$ be a knot whose exterior is not a solid torus and such that a non-trivial surgery on it gives $L$. Then the exterior of $K$ is not Seifert fibred, $p|w^2$ and the surgery slope, $n$, is an integer that satisfies the following (with some choice of sign):
$$
n = -q\frac{w^2}{p}\pm 1.
$$

Moreover, there is at most one such slope (i.e. we can choose either $+$ or $-$ but not both in the equation above).
\label{cor-lens-space-complements}
\end{cor}

\begin{proof}
By Theorem \ref{theorem-unknot-characterisation} we only need to consider non-null-homologous knots. Let $L = L(p,q)$ be a lens space and $K$ a non-null-homologous knot in it.

Suppose the exterior of $K$ is not Seifert fibred. Then by the Cyclic Surgery Theorem \cite{CGLS} the slope has to be integral.

We can isotope $K$ into one of the Heegaard solid tori $W$ of $L$. Then we can get $L$ by first performing an integral surgery on $K$ in $W$ and then glueing the other solid torus from the outside so that its meridian becomes the $(p,q)$-curve.

Let $\mu$ be the meridian of $W$, fix a curve $\lambda$ in $\partial W$ with $\lambda \cdot \mu = 1$ and embed $W$ into $S^3$ in the standard way with respect to $\mu$ and $\lambda$. This endows $K$ with a well-defined longitude $l$. Let $m$ be the meridian of $K$. Suppose $K$ has winding number $w$ in $W$. Then in $H_1(W\setminus K)$ (which is generated by $m$ and $\lambda$) $l = w\lambda$ and $\mu = wm$. Let $n$ be the surgery slope with respect to these coordinates. Then surgery on $K$ introduces a relation $nm+w\lambda = 0$. The other Heegaard solid torus introduces a relation $-qwm+p\lambda$. All in all, the first homology of $L$ has presentation matrix
$$
\begin{pmatrix} n & w \\ -qw & p \\ \end{pmatrix}.
$$

The order of the first homology of $L$ is the absolute value of the determinant of the relation matrix. Thus we must have $\pm p = np+qw^2 \Rightarrow n = -q\frac{w^2}{p} \pm 1$. However, $q$ is coprime with $p$, so $p|w^2$. By the Cyclic Surgery Theorem the distance between slopes that that give lens spaces is at most one, so $-q\frac{w^2}{p} + 1$ and $-q\frac{w^2}{p} - 1$ cannot both produce a lens space.

If $p$ is square-free $p|w^2$ implies that $p|w$, so $K$ is null-homologous in $L$ -- a contradiction.

Now suppose the exterior of $K$ is Seifert fibred. As stated, Conjecture \ref{conj-complement} has been proven for knots with Seifert fibred exteriors in \cite[Theorem 1]{rong}.

However, to demonstrate that there are no non-trivial slopes equivalent to the meridian we will provide a different proof.

By \cite[Lemma 2]{rong} we can assume that $K$ is a fibre in some fibration of $L$.

Fibrations of lens spaces come in two families (see e.g. \cite[Theorem 2.3]{hatcher3mflds}). All lens spaces can be fibred over a sphere with at most two exceptional fibres. There are also some lens spaces that can be fibred over the projective plane with one exceptional fibre of invariant $(n,1)$.

In the first (and most common) case, to avoid the exterior being a solid torus, $K$ must be an ordinary fibre in a fibration with two exceptional fibres. Suppose the invariants of these fibres are $(p_1,x_1)$ and $(p_2,x_2)$. Since Seifert fibred spaces with three exceptional fibres are not lens spaces, surgery on $K$ must introduce a fibre with invariant $(1,n)$. If so, the order of homology changes from $|p_1x_2+p_2x_2|$ to $|p_1x_2+p_2x_2+np_1p_2|$. Equality is only possible if $p_1=p_2=2$, i.e. $L$ is obtained by filling the Seifert fibred space over a disc with two exceptional fibres with the same invariant $(2,1)$.

Different fillings of this space that produce lens spaces can be indexed by integers $m\in \ZZ$ and they produce $L(4m,2m+1)$. If $L(4m,2m+1) = L(4n,2n+1)$ for $m\neq n$ (equality sign here means `there exists an orientation-preserving homeomorphism'), then $m=-n$, i.e. the two lens spaces are $L(4m,2m\pm 1)$. If they are to be homeomorphic by an orientation preserving homeomorphism, then we must have $(2m+1)(2m-1) \equiv 1\ (\mbox{mod } 4m)$ (see \cite{przytyckiYasukhara} and references therein) which is not true (even though this spaces \slshape are \upshape homeomorphic by an orientation-reversing homeomorphism).

This deals with the case when $K$ is a fibre of a fibration of $L$ over a sphere. $L$ may also have a fibration with one exceptional fibre over the projective plane in which case the invariant of the exceptional fibre is $(n,1)$. In this case $L = L(4n,2n+1)$. If $K$ is the exceptional fibre, then the only non-trivial surgery which still gives a lens space yields $L(4m,2m+1)$ for $n\neq m$. This case has been dealt with above. Similarly, if $K$ is an ordinary fibre then surgeries on it are indexed by an integer $k$ and their effect is to change the invariant of the exceptional fibre to $(n, nk+1)$. It follows that we must have $nk+1 = -1$, so surgery again gives $L(4n, 2n-1)$.
\end{proof}

\bibliographystyle{alpha}  
\bibliography{biblio}

\end{document}